\documentclass[12pt]{amsart}

\usepackage{amsfonts,amssymb,amsmath,amsthm,amscd,graphicx,float}

\newtheorem{theorem}{Theorem}[section]
\newtheorem{lemma}[theorem]{Lemma}
\newtheorem{prop}[theorem]{Proposition}

\newtheorem{cor}[theorem]{Corollary}
\newtheorem{scholium}[theorem]{Scholium}

\theoremstyle{definition}
\newtheorem{definition}[theorem]{Definition}

\theoremstyle{remark}

\numberwithin{equation}{section}

\newcommand{\lcr}{\raisebox{-5pt}{\mbox{}\hspace{1pt}
                 \includegraphics{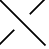}\hspace{1pt}\mbox{}}}
\newcommand{\ift}{\raisebox{-5pt}{\mbox{}\hspace{1pt}
                 \includegraphics{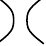}\hspace{1pt}\mbox{}}}
\newcommand{\zer}{\raisebox{-5pt}{\mbox{}\hspace{1pt}
                 \includegraphics{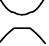}\hspace{1pt}\mbox{}}}

\title{The Structure of the Kauffman Bracket Skein Algebra at Roots of Unity }

\author{Charles Frohman }
\address{ Department of Mathematics, The University of Iowa}
\email{charles-frohman@uiowa.edu}
\author{Joanna Kania-Bartoszynska}
\address{ Division of Mathematical Sciences, The National Science Foundation}
\email{jkaniaba@nsf.gov}
\thanks{This material is based upon work supported by and while serving at the National Science Foundation. Any opinion, findings, and conclusions or recommendations expressed in this material are those of the authors and do not necessarily reflect the views of the National Science Foundation. This work was conducted while  JKB was on sabbatical leave at the Department of Mathematics of the University of Iowa. She thanks the department for their hospitality.}

\begin{document}
\begin{abstract} 
This paper is focused on the structure of the Kauffman bracket skein algebra of a punctured surface at roots of unity.  A criterion that determines when a collection of skeins forms a basis of the  skein algebra as an extension over the $SL(2,{\mathbb C})$ characters of the fundamental group of the surface, with appropriate localization is given. This is used to prove that when the algebra is localized so that every nonzero element of the center has a multiplicative inverse, that it is a division algebra.  Finally, it is proved that the localized skein algebra can be split over its center as a tensor product of two commutative subalgebras.
\end{abstract}
\maketitle
\section{Introduction}

In this paper we study the structure of the Kauffman bracket skein algebra of an oriented punctured surface at a $2N$-th  root of unity where $N$ is odd. These algebras play a crucial role in constructions of quantum invariants of $3$-manifolds and in investigating quantum hyperbolic geometry.

Given an orientable surface $F$ its Kauffman bracket skein algebra $K(F)$ is formed by taking linear combinations of framed links in a cylinder over the surface, $F\times I$, with complex coefficients, and modding out by the relations that define the Kauffman bracket link invariant. The multiplication comes from stacking one link above the other, with the up and down direction given by the interval $I$. The Kauffman bracket skein relation involves a complex parameter $A$.  We assume  that $A$ is a $2N$-th root of unity for odd $N$, and use the notation $K_N(F)$ to indicate the level. It was shown by Bullock \cite{B}, and Przytycki-Sikora \cite{PS} that when $N=1$, so that $A=-1$, this algebra is isomorphic to the character ring of the $SL(2,{\mathbb C})$ representations of the fundamental group of the surface. Bonahon and Wong  \cite{BW1} constructed an embedding of $K_1(F)$  into the center of $K_N(F)$ via the so-called threading map,  showing that the algebra $K_N(F)$ is a central extension of the $SL(2,{\mathbb C})$-characters of $F$.  Abdiel and Frohman \cite{AF2} proved that  this is a finite extension. They localized the algebra over the non-zero characters of $F$ and showed that the localized algebra is finite dimensional over the function field of the $SL(2,{\mathbb C})$-character variety of $\pi_1(F)$.

In this paper we work with  surfaces of negative Euler characteristic that have at least one puncture. The Kauffman bracket skein algebra has no zero divisors, and is finite rank over its center. The center of $K_N(F)$ is the coordinate ring of an algebraic variety that is a finite sheeted branched cover of the character variety of the surface. In algebraic geometry, localization corresponds to restricting to an open subset of the underlying space. The localizations in this paper come from restricting to Zariski open subsets of the variety related to the center of $K_N(F)$. We define a local basis of the Kauffman bracket skein algebra, as a module over some subring of the center, as a collection of linearly independent skeins that becomes a basis after inverting an element of the subring of the center. We give a criterion for when a collection of skeins is a local basis. An immediate consequence of our criterion is that $K_N(F)$ localized so that every nonzero character is invertible is a vector space of dimension $N^{-3e(F)}$, where $e(F)$ is the Euler characteristic of $F$, and the base field is the function field of the character variety.

The primary application of our criterion for detecting a local basis is the construction of a splitting of a localization of the  skein algebra as a tensor product of two commutative subalgebras over its center. This splitting is generated by  two pants decompositions of the surface $F$, i.e. collections of disjoint simple closed curves that cut the surface into a collection of pairs of pants (thrice punctured spheres).  Unless the parameter $A$ is $\pm 1$, the Kauffman bracket skein algebra of a surface of genus greater than $1$ is non-commutative. The intersections between the two collections of curves yielding the two pants decompositions involved in the splitting completely determines the non-commutativity of $K_N(F)$.

 The paper is organized as follows. We start in Section \ref{review} by recalling the definitions and known facts about the Kauffman bracket skein algebras, including the relationship between skeins and admissible colorings of a triangulation of the surface. In Section \ref{basis} we formulate and prove a criterion that characterizes when a collection of skeins forms a basis of an appropriately localized Kauffman bracket skein algebra over an extension of the localized characters, and compute the dimension of such extensions.
 
In order to state the main results we need some terminology.  We denote the image of the threading map  in $K_N(F)$ by $\chi(F)$, to remind the reader that it is isomorphic to the ring of $SL(2,\mathbb{C})$-characters of the fundamental group of $F$.  An admissible coloring of the edges of an ideal triangulation $E$ of a surface $F$ is an $E$-tuple of nonnegative integers that satisfy parity conditions and triangle inequalities.  The admissible colorings are in one to one correspondence with the isotopy classes of simple diagrams on the surface. The {\bf residue} of a coloring is the tuple of remainders of its values under division by $N$ in $\mathbb{Z}_N^E$.  Every skein is a linear combination of simple diagrams. The simple diagrams are linearly ordered by a choice of an ordering of the edges $E$.  Every skein has a {\bf lead term}, involving the largest simple diagram. The {\bf residue of the lead term} of a skein is the residue of the admissible coloring underlying its lead term. We say that a collection of skeins $\mathcal{B}$ is a local basis for $K_N(F)$ if there is $c\in \chi(F)$ such that if $\chi(F)_c\leq K_N(F)_c$ are the localizations corresponding to inverting the powers of $c$, then $K_N(F)_c$ is a free module over $\chi(F)_c$ with basis given by the images of the skeins in $\mathcal{B}$.

 \vspace{.1in}
{\bf {Scholium [Exhaustion] \ref{exhaustion}:}}{\it{ If $\mathcal{B}$ is any collection of skeins such that the residues of the lead terms of the elements of $\mathcal{B}$ are in one to one correspondence with the elements of $\mathcal{Z}_N^E$, then $\mathcal{B}$ is a local basis for $K_N(F)$. }}

\vspace{.1in}

\vspace{.1in}
{\bf Corollary \ref{dimKoverchi} }: {\it{The dimension of $S^{-1}K_N(F)$ as a vector space over $S^{-1}\chi(F)$ is $N^{-3e(F)}$, where $e(F)$ is the Euler characteristic of surface $F$.}}
 \vspace{.1in}
 
  \vspace{.1in}
{\bf{Corollary \ref{division}:}} {\it{The Kauffman bracket skein algebra $K_N(F)$, localized so that every nonzero character is invertible, is a division algebra.}}
 \vspace{.1in}
 
  In Section \ref{pants} we construct a decomposition of the localized Kauffman bracket skein algebra of a punctured surface as a tensor product over the localized center of two commutative subalgebras generated by specifically chosen pants decompositions of the surface.
  
   \vspace{.1in}
  {\bf{Theorem \ref{pantssplit}}}:
{\it{ Let $F$ be an orientable surface of genus $g$ with $p$ punctures and negative Euler characteristic, where $p\geq 1$.  There exist two pants decompositions $P$ and $Q$ of $F$ and a skein $c\in Z(K_N(F)$ such that the associated skein modules 
 ${\mathcal P}$ and ${\mathcal Q}$ localized over $S=\{c^k\}$ give a splitting of the localized skein algebra of $F$ over its center.
 \begin{equation}
K_N(F)_c  = \mathcal{P}_c\otimes_{Z(K_N(F))_c}\mathcal{Q}_c.
 \end{equation}
 The intertwiner is given by $p\otimes q \rightarrow p*q$.}}
  \vspace{.1in}

The inspirations for this theorem are \cite{AF1}, where it is proved for the skein algebra of a surface of genus one with one puncture, a paper of Natanzon, and Felikson \cite{NF}, where they produce local coordinates for Teichm\"{u}ller space using pairs of pants decompositions, and a lecture of  Bonahon where he discussed representations of the skein algebra.

\section{ A review of the Kauffman bracket skein algebra at roots of unity}\label{review}
In this section we review the definitions and known results about the Kauffman bracket skein algebra that we need in subsequent sections.

 If $M$ is an oriented three manifold, let $\mathcal{L}(M)$ be the set of all isotopy classes of oriented framed links in $M$, including the empty link. Recalling that a framed link is an embedding of  a disjoint union of annuli in $M$, the orientation of a component corresponds to choosing a preferred side to the annulus. In diagrams we draw  immersed  arcs or circles with overcrossing data. You should think of part of the annulus in the diagram to be a band in the page parallel to the arc, with the preferred side up.

Recall the Kauffman bracket skein relations, 

\begin{equation}\label{KBSR}
\lcr-A\zer-A^{-1}\ift \end{equation}
and
\begin{equation*}\bigcirc \cup L+(A^2+A^{-2})L,\end{equation*} 
 where the framed links in 
each expression are
identical outside the balls pictured in the diagrams. Let $N$ be an odd natural number and let $A=e^{\pi{\bf i}/N}$.  The {\bf Kauffman bracket skein module}, $K_N(M)$, is the quotient of the vector space with basis $\mathcal{L}(M)$ and complex coefficients by the subvector space spanned by all Kauffman bracket skein relations.

For any  oriented surface $F$, the skein module  $K_N(F\times [0,1])$ is an algebra under stacking. More precisely, product of two links is defined by placing one link above the other in the direction given by the interval $[0,1]$. The product descends distributively to a product on the skein module. We denote the product of skeins $\alpha$, $\beta$ by $\alpha *\beta$. Often, we will know apriori that the skeins being multiplied commute, for instance if $J_1$ and $J_2$ are disjoint simple closed curves.  If this is the case we dispense with the star, and denote the product by juxtaposition, $J_1J_2$.

It can be the case that $F$ and $F'$ are not homeomorphic, but  $F\times [0,1]$ is homeomorphic to $F'\times [0,1]$. Nevertheless the algebras coming from stacking in $K_N(F\times [0,1])$ and $K_N(F'\times [0,1])$ are not isomorphic. To emphasize that the algebra  comes from the product structure we denote the algebra $K_N(F)$. When $N=1$ the algebra $K_1(F)$ is commutative, and can be canonically identified with the coordinate ring of the $SL_2\mathbb{C}$-character variety of the fundamental group of $F$.

\begin{theorem}[\cite{B,PS}]\label{Dougs}
Given an orientable surface $F$, let $X(\pi_1(F))$ denote the ring of $SL_2{\mathbb{C}}$-characters of $\pi_1(F)$. The map
 \begin{equation} \theta:K_{1}(F)\rightarrow X(\pi_1(F))\end{equation} that takes each knot $K$ to $-tr(\gamma)$ where $\gamma\in \pi_1(M)$ is a loop corresponding to $K$ is an isomorphism.

\end{theorem}

A {\bf simple diagram} on a surface $F$ is a system of disjoint simple closed curves, none of which bounds a disk in $F$.  A simple diagram is {\bf primitive} if no two curves in the diagram cobound an annulus in $F$. Corresponding to a simple diagram is a framed link obtained by choosing a system of disjoint annuli in $F\times \{0\}$ that form a regular neighborhood of the diagram, oriented so that the preferred side is upwards. This is called the {\bf blackboard framing} of the simple diagram.  If two simple diagrams are isotopic, so are their corresponding framed links.  The Kauffman bracket skein relations (\ref{KBSR}) allow us to resolve all crossings and get rid of any simple closed curves that bound  disks in $F$. These observations yield the following:

\begin{prop}\label{simplediag}The set of framed links coming from isotopy classes of simple diagrams forms a basis for $K_N(F)$ as a vector space over $\mathbb{C}$.
\end{prop}

Recall that if $\alpha$ and $\beta$ are properly embedded $1$-manifolds in a surface $F$, and at least one of $\alpha$ or $\beta$ is compact, their {\bf geometric intersection number} $i(\alpha,\beta)$  is the minimum number of points in $\alpha'\cap \beta'$ over all properly embedded $1$-manifolds $\alpha'$ and $\beta'$ that are isotopic to $\alpha$ and $\beta$ via a compactly supported isotopy. The geometric intersection number can always be realized by transverse representatives of the isotopy classes of $\alpha$ and $\beta$. Also if $\alpha$ and $\beta$ are transverse, a {\bf bigon} is a disk $D\subset F$ such that the boundary of $D$ is the union of two arcs, $a\subset \alpha$ and $b\subset \beta$. A pair of transverse representatives of the isotopy classes of $\alpha$ and $\beta$ realize $i(\alpha,\beta)$ if and only if they have no bigons.

A surface $F$ has {\bf finite type} if it is the result of removing finitely many points from a closed oriented surface. An {\bf ideal triangle} is the result of removing three points from the boundary of a disk.  The three open intervals that are the complement of the three points removed from the boundary are the {\bf sides} of the ideal triangle. An {\bf ideal triangulation} of a surface $F$ is given by a collection $\Delta_i$ of ideal triangles whose sides have been identified in pairs to obtain a quotient space $X$, along with a homeomorphism $h:X\rightarrow F$. For each triangle there is an inclusion map $\Delta_i\rightarrow F$. We say that the triangle $\Delta_i$ is {\bf folded} if the inclusion map is not an embedding. If this is the case, then there are two sides of $\Delta_i$ that are identified to each other.  The images of the sides of the $\Delta_i$ in $F$ are called the {\bf edges of the triangulation}. For a surface of finite type to admit an ideal triangulation its Euler characteristic $e$ must be negative. In this  case there are $-2e$ triangles  and $-3e$ edges in the triangulation. 

A function $f:E\rightarrow \mathbb{Z}_{\geq 0}$ is {\bf admissible} if:
\begin{itemize} 
\item When $a,b,c$ are three distinct edges that are the sides of an ideal triangle, then $f(a)+f(b)+f(c)$ is even, and the integers $f(a),f(b),f(c)$  satisfy all triangle inequalities;
\item When $a,b$ are the edges of a folded triangle, where $a$ is the image of two sides then  $f(b)$ is even, and $2f(a)\geq f(b)$. 
\end{itemize}

\begin{prop}
Given an ideal triangulation of the surface $F$ with edges $E$, the isotopy classes of simple diagrams on  $F$ are in one to one correspondence with admissible colorings of $E$. 
If $S$ is a simple diagram, we denote the coloring corresponding to $S$ by $f_S$. If $f:E\rightarrow \mathbb{Z}_{\geq 0}$ is an admissible coloring then we denote the isotopy class of simple diagrams corresponding to $f$ by $[f]$. 
\end{prop}
\begin{proof}
Given an admissible coloring, there is a unique isotopy class of simple diagrams, so that the geometric intersection number of a transverse diagram $S$ in the isotopy class with each edge $a\in E$ is $f(a)$. Furthermore, if $S$ is a transverse simple diagram  then $f(a)=i(a,S)$ defines an admissible coloring of $E$. 
\end{proof}

Note that $S\in [f_S]$. A simple diagram that realizes its geometric intersection number with the edges of the triangulation is said to be in {\bf normal position}. 

The set of isotopy classes of simple diagrams is well ordered via its corresondence with admissible colorings of the triangulation.
\begin{definition}\label{wellorder}
Given an ideal triangulation of the surface $F$ with edges $E$ choose an ordering of  $E$. Each admissible coloring of $E$ can then be represented as a $(-3e)$-tuple of nonnegative integers.  Order these tuples lexicographically to get a well ordering of admissible colorings.  
\end{definition}

The sum of admissible colorings is admissible, and the set of admissible colorings forms a well ordered monoid under this operation. This well ordering gives rise to a well ordering of isotopy classes of simple diagrams. In this ordering the smallest admissible coloring is the coloring that assigns $0$ to each edge.
\begin{definition}\label{leadterm}   If $\alpha\in K_N(F)$ is a skein, we can write $\alpha$ as a linear combination with complex coefficients of finitely many simple diagrams.  Define the {\bf lead term} of $\alpha$ to be 
the term that involves the largest simple diagram with nonzero coefficient, and denote it $ld(\alpha)$.  
\end{definition}
It was proved in \cite{AF2} that lead terms behave well with respect to multiplication of skeins.

\begin{theorem}[\cite{AF2}]\label{lead} If $aS$ is the lead term of the skein $\alpha$ and $bS'$ is the lead term of the skein $\beta$, where $a,b\in \mathbb{C}$ and $S$ and $S'$ are simple diagrams, then there exists an integer $n$ such that the lead term of $\alpha*\beta$ is $A^nab[f_S+f_{S'}]$. \end{theorem}

 Define the {\bf Chebyshev polynomials of the first kind} by $T_0=2$, $T_1=x$ and recursively by $T_k=xT_{k-1}-T_{k-2}$. In closed form, \begin{equation}\label{tcheb}T_k(x)=\sum_{i=0}^{\lfloor{k/2}\rfloor}(-1)^i\frac{k}{k-i}\binom{k-i}{i}x^{k-2i}.\end{equation}
The Chebyshev polynomials satisfy three important identities: 
\begin{itemize}  
\item The {\bf product to sum formula},
\begin{equation} T_a(x)T_b(x)=T_{a+b}(x)+T_{|a-b|}(x);\end{equation}
\item The {\bf multiplicative property},
\begin{equation}T_a(T_b(x))=T_{ab}(x);\end{equation}
\item  The {\bf DeMoivre's formula},
\begin{equation}
T_k(z+z^{-1})=z^k+z^{-k}.\end{equation}
\end{itemize}
In fact, $T_k(x)$ is the polynomial such that $T_k(2\cos{\theta})=2\cos{k\theta}$.

Let $x$ denote a simple closed curve on a twice punctured sphere $Ann$ that is a deformation retract of $Ann$. The structure of the Kauffman bracket skein algebra  of $Ann$ is well known.
\begin{prop}
The algebra $K_N(Ann)$ is isomorphic to $\mathbb{C}[x]$, polynomials in $x$ with complex coefficients. 
\end{prop}
\begin{proof}The isomorphism is given by sending the skein $x$ to the variable $x$, since the simple diagrams on $Ann$ are isotopic to parallel copies of $x$. 
\end{proof}
\begin{cor}
Chebyshev polynomials $\{ T_k(x) | k\in {\mathbb Z}_{\geq 0}\}$ form a basis for $K_N(Ann)$ over the complex numbers.
\end{cor}
\begin{proof}
This follows from the fact that they have the same leading terms as the standard basis $\{x^k |k\in {\mathbb Z}_{\geq 0}\}$.
\end{proof}

Since a framed link is a disjoint union of oriented annuli embedded in a three-manifold $M$,   we can {\bf thread} the components of a  link with copies of $T_k(x)$, using the annulus as a guide.  It is a theorem of Bonahon and Wong \cite{BW1}, that  when all components of links are threaded with $T_N(x)$  this operation gives a map $\tau:K_1(M)\rightarrow K_N(M)$. Furthermore, if a component of a link is threaded with $T_N(x)$ then you can change its crossings with any other component, without changing the skein.  In the case where $F$ is a surface they prove the following: 

\begin{theorem}[\cite{BW1}]\label{thread}
The threading  map 
\begin{equation} \tau:K_1(F)\rightarrow K_N(F)\end{equation} is an embedding of $K_1(F)$ into the center of $K_N(F)$. 
\end{theorem}

A purely skein theoretic proof of this theorem was given by L\^{e} in \cite{Le}.
Our viewpoint is  that  $K_N(F)$ is a central extension of $K_1(F)$.
For that reason the image of $\tau$ is denoted by $\chi(F)$, to remind the reader that it is canonically isomorphic to the coordinate ring of the $SL_2\mathbb{C}$-character variety of the fundamental group of $F$, see Theorem \ref{Dougs}.  The rings $K_1(F)$ and $K_N(F)$ are filtered by leading terms. Notice that  $\tau$ maps the filtered submodules of $K_1(F)$ monotonically to the filtered submodules of $K_N(F)$.

In \cite{AF2} it is  proved that:
\begin{theorem}\label{spanning} If $F$ is a surface of finite type,  then there exist simple closed curves $J_1,\ldots,J_n$ on $F$  so that the collection of skeins $T_{k_1}(J_1)*T_{k_2}(J_2)*\ldots*T_{k_n}(J_n)$, where $T_{k_i}(J_i)$ denotes the result of threading $J_i$ with the Chebyshev polynomial $T_{k_i}(x)$  and where the $k_i$ range over all natural  numbers, spans $K_N(F)$ as a vector space over the complex numbers. \end{theorem}
This  theorem has an important corollary:

\begin{cor}[\cite{AF2}] \label{finite} For any finite type surface $F$, the module $K_N(F)$ is a finite central extension of $\chi(F)$. More specifically, if $J_i$ are the curves from Theorem \ref{spanning} then the skeins  $T_{k_1}(J_1)*T_{k_2}(J_2)*\ldots*T_{k_n}(J_n)$  where the $k_i$ range over $\{0,\ldots,N-1\}$ span $K_N(F)$ as a module over $\chi(F)$.\end{cor}

  Let 
\begin{equation}\mathcal{A}=\{f:E\rightarrow \mathbb{Z}_{\geq 0}| f \ \mathrm{is} \ \mathrm{admissible}\}\subset \mathbb{Z}_{\geq 0}^E.\end{equation} 
There is a quotient map, \begin{equation}res:\mathcal{A}\rightarrow \mathbb{Z}_N^E,\end{equation} that sends each  admissible coloring to the  $E$-tuple in $\mathbb{Z}_N$ consisting of the residue classes of its values 
modulo $N$.
\begin{prop} For odd natural numbers $N$ the map $res:\mathcal{A}\rightarrow \mathbb{Z}_N^E$  is onto. \end{prop}

\proof An element of $\mathbb{Z}_N^E$ can be thought of as a function $g:E\rightarrow \mathbb{Z}_{\geq 0}$ that takes on values between $0$ and $N-1$. Define a second function $f:E\rightarrow \mathbb{Z}_{\geq 0}$ as follows.  If $g(a)$ is odd, let $f(a)=g(a)+3N$. If $g(a)$ is even let $f(a)=g(a)+2N$.  Notice that all the values of $f$ are even, so for any $a,b,c$ that are the sides of an ideal triangle $f(a)+f(b)+f(c)$ is even. If $a,b$ are the sides of a folded triangle, and $a$ is the doubled edge, then $f(b)$ is even.  

If $a,b,c$ are the sides of a triangle, then $f(a)+f(b)\geq 4N\geq f(c)$. If $a,b$ are the sides of a folded triangle and $a$ is the doubled edge then $2f(a)\geq 4N\geq f(b)$.   This means that $f\in \mathcal{A}$.  

Finally notice that $res(f)=g$.  Since $g$ was arbitrary, $res:\mathcal{A}\rightarrow \mathbb{Z}_N^E$ is onto. \qed

Once a triangulation  of $F$ has been chosen, and the edges of the triangulation have been ordered, each nonzero skein $\alpha\in K_N(F)$ has a  lead term $ld(\alpha)=z[f]$ where $z \in \mathbb{C}^*$ and $f\in \mathcal{A}$ as in the Definition \ref{leadterm}. We define  $res(\alpha)$ to be $res(f)$.

In \cite{AF2} it is proved: 

\begin{prop}[\cite{AF2}] \begin{itemize}
\item A simple diagram $S$ is the lead term of an element of $\chi(F)$ if and only if the residue of its associated admissible coloring $res(f_S)$ is the $0$-tuple in $\mathbb{Z}_N^E$.
\item If $\{S_i\}$ is a collection of simple diagrams such that the residues of their associated colorings $res(f_{S_i})$ are distinct, then the skeins corresponding to the $S_i$ are independent over $\chi(F)$.\end{itemize}\end{prop}

\section{Bases}\label{basis}
Our goal in this section is to prove a criterion that determines when a set of simple diagrams forms a local basis of $K_N(F)$.
We start by recalling some needed definitions from ring theory.

Rings are assumed to be associative with unit.  The {\bf center} $Z(R)$ of a ring $R$ is,
\begin{equation} Z(R)=\{a\in R| \forall r\in R \ ar=ra\}.\end{equation}  
Given a commutative ring $C$ contained in the center of a ring $R$ we say that $R$ is a {\bf central ring extension} of $C$. Ring  $R$ has {\bf no zero divisors} if for all $a,b\in R$, $ab=0$ implies $a=0$ or $b=0$. An element $a\in R$ is {\bf nilpotent} if for some $n$, $a^n=0$. Notice that if $R$ has no zero divisors, then it has no nilpotents. A ring $R$ is a {\bf division algebra} if for every nonzero  $a \in R$ there exists $b\in R$ with $ba=1$. 

Let $R'$ be any subring of the ring $R$ (not necessarily proper). If 
$S$ is a multiplicatively closed subset  of the center of $R'$ that does not contain $0$ then we can form the {\bf localization} of $R'$ at $S$, denoted $S^{-1}R'$, to be  the set of equivalence classes of pairs in $S\times R'$ defined as follows, 
\begin{equation}
S^{-1}R'= \{(a,s)\in R'\times S\}/\{(a,s)\sim (b,t)\iff at=bs\},
\end{equation}
 where addition and multiplication are given by
\begin{equation} [a,s]+[b,t]=[at+sb,st], \ \mathrm{and} \ [a,s][b,t]=[ab, st].\end{equation}
Note that any element of $S$ has a multiplicative inverse in the localization $S^{-1}R'$.  If $C\subset R$ is a central ring extension and $S\subset C$ then $S^{-1}C\subset S^{-1}R$.
If $S=\{c^k\}$, that is if the multiplicatively closed subset consists of the nonnegative powers of a central element $c\in R'$, then the localization is denoted by $R'_c$.

\begin{definition}  
Suppose that $C\subset R$ is a central extension and 
 $R$ is a free module over $C$.  We say that $\mathcal{B}=\{r_1,\ldots,r_k\}$ is a {\bf local basis} of $R$ if there exists $c\in C$ such  that the multiplicatively closed set $S=\{c^k\}$ does not contain zero, and the image of the elements of $\mathcal{B}$ in $R_c$ form a basis for $R_c$ over $C_c$.\end{definition}  
 For instance the Gaussian integers $\mathbb{Z}[{\bf i}]$ are a central ring extension of the integers  that is free of rank $2$. Notice that $\mathcal{B}=\{2,{\bf{i}}\}$ is a local basis. If we invert the powers of $2$, the image of  $\mathcal{B}$ is a basis of the resulting localization of the Gaussian integers over the the localization of the integers.

We say that a bilinear pairing $T:R\otimes_C R \rightarrow C$ is {\bf symmetric}
if $T(r\otimes s)=T(s\otimes r)$ for all $r,s\in R$. We say $T$ is {\bf  nondegenerate} if
for every nonzero $r\in R$ there exists $s\in R$ with $T(r\otimes s)\neq 0$.  We say the submodule $V\subset R$ {\bf exhausts} $T$ if for every nonzero $r\in R$ there exists $v\in V$ with $T(v\otimes r)\neq 0$.  

\begin{theorem} \label{exhaustbasis}Suppose that $C\subset R$ is a central extension and that $T:R\otimes_C R \rightarrow C$ is a nondegenerate symmetric bilinear pairing. Let $B=\{v_1,\ldots, v_n\}$ be a subset of $R$ and let $(g_{ij})$ be the matrix defined by
\begin{equation}  g_{ij}=(T(v_i\otimes v_j)).\end{equation}
If $B$ exhausts $T$, and $det(g_{ij})$ is not nilpotent in $C$, then $B$ is a local basis for $R$. \end{theorem}

\proof  
Note first that if $c=det(g_{ij})$ is not nilpotent in $C$ then we can localize to form $C_c$ and  $R_c$,
and the matrix $g=(g_{ij})$ is invertible in $M_n(C_c)$. Define $g^{ij}\in C_c$ to be the entry in the $ith$ row and $j$th column of $g^{-1}$. Let $w_j=\sum_ig^{ij}v_i$.  Since $gg^{-1}$ 
is the identity matrix we have $T(v_i\otimes w_j)=\delta_i^j$, (Kronecker's delta).

We need to show that $B$ is a basis for $R_c$ as a module over $C_c$. This means showing that it is independent and it spans.  Since $det(g)\neq 0$, the set $B$ is independent.
If $w\in R$, consider $w-\sum_jT(v_j\otimes w)w_j$.  Notice that $w-\sum_jT(v_j\otimes w)w_j$ pairs to zero with every $v_i$. Since the set $B$ is exhaustive this implies that $w-\sum_i T(v_j\otimes w_j)w_j=0$, but that means $B$ spans.  \qed

 Bringing this back to the current discussion, we want to apply the Theorem \ref{exhaustbasis} in the case when we work with the Kauffman Bracket skein algebra of a surface. We need to know that it has no nilpotents. This follows from the following more general fact.
  \begin{prop}
 Let F be an oriented finite type surface. The algebra $K_N(F)$ has no zero divisors.

 \end{prop}
 \proof
 This  follows immediately from the fact that the lead term of a product of two skeins cannot cancel with any other term, see Theorem  
 \ref{lead}.
 \qed
 
 We  now describe a nondegenerate symmetric bilinear pairing,
\begin{equation} T:K_N(F)\otimes_{\chi(F)} K_N(F)\rightarrow \chi(F). \end{equation}

Any simple diagram $D$ on a surface $F$ is made up of collections of parallel simple closed curves. If the simple diagram consists of $k_i$ parallel copies of the simple closed curves $C_i$, then the diagram that is  the union of the single copies of each  $C_i$ is primitive.  We can then form the skein $\prod_iT_{k_i}(C_i)$ which we call a {\bf threaded primitive diagram}.  The lead term of $\prod_iT_{k_i}(C_i)$ is the simple diagram $D$. 
This gives rise to the following consequence of Proposition \ref{simplediag}.
\begin{prop}\label{threadedbasis}
The set of threaded primitive diagrams is a basis for $K_N(F)$ over $\mathbb{C}$. 
\end{prop}
\proof
 There is an inductive process for rewriting a linear combination of simple diagrams as a linear combination of threaded primitive diagrams.  Given a skein $\alpha$ its lead term is a simple diagram $D$. Let $\prod_iT_{k_i}(C_i)$ be the corresponding threaded primitive diagram.
Subtract and add $\prod_iT_{k_i}(C_i)$ from $\alpha$ and let $D$ cancel with the lead term of the subtracted $\prod_iT_{k_i}(C_i)$. Now you have rewritten the skein so that largest term that isn't rewritten in terms of threaded primitive diagrams is smaller. Do the same with the next largest term. This process can be continued until the skein has been written as a linear combination of threaded primitive diagrams. The threaded primitive diagrams are linearly independent as their lead terms are a basis for $K_N(F)$. \qed

 \begin{cor}\label{chi} If the lead term of  the skein $\alpha\in K_N(F)$ is a simple diagram whose residue is  $\vec{0}$, then the lead term of the skein written as a linear combination of threaded primitive diagrams is in $\chi(F)$. \end{cor}

\proof The rewriting process produces an element of $\chi(F)$ as the lead term. Subsequent steps produce lower order terms that cannot cancel with it. \qed

In \cite{AF2} a map $Tr:K_N(F)\rightarrow \chi(F)$ is defined by writing a skein as a complex linear combination of threaded primitive diagrams using Proposition \ref{threadedbasis}.  Erase any term, where there is a curve threaded by $T_k(x)$ where $N$ doesn't divide $k$.  To be more specific,  suppose
\begin{equation}\alpha=\sum_{k_1,\ldots,k_n}\alpha_{k_1,\ldots,k_n}\prod_iT_{k_i}(C_i)\end{equation}
where the $\alpha_{k_1,\ldots,k_n}\in \mathbb{C}$, and the $\prod_iT_{k_i}(C_i)$ are threaded primitive diagrams, then $Tr(\alpha)$ is the sum of all the terms where $N|k_i$ for all $i$. 

The map $Tr:K_N(F)\rightarrow \chi(F)$ has the following properties:
\begin{itemize}
\item It is $\chi(F)$-linear; 
\item It is cyclic, in the sense that $Tr(\alpha*\beta)=Tr(\beta*\alpha)$; 
\item It satisfies $Tr(1)=1$; 
\item  It is nondegenerate in the sense that for any $\alpha \in K_N(F)$  if $\alpha\neq 0$ then there exists $\beta \in K_N(F)$ with $Tr(\alpha*\beta)\neq 0$.  
\end{itemize}

\begin{prop}
The Kauffman bracket skein algebra $K_N(F)$ is a finite rank module over $\chi(F)$,
and the pairing 
\begin{equation} \label{tracepairing}
T:K_N(F)\otimes_{\chi(F)}K_N(F)\rightarrow \chi(F) \end{equation}
given by $T(\alpha\otimes \beta)=Tr(\alpha*\beta)$ is symmetric and nondegenerate.
\end{prop}
\proof The fact that  $K_N(F)$ is a finite rank module over $\chi(F)$ follows from the Corollary \ref{finite}. Since the  map $Tr$ is cyclic,   the pairing $T$ is symmetric. We re-iterate that it is non-degenerate in the Lemma \ref{nondegenerate} below. \qed 

We say two skeins $\alpha$ and $\beta$ have {\bf complementary residues} if the residues of their lead terms sum to zero in $\mathbb{Z}_N^E$.

\begin{lemma} \label{nondegenerate} If $\alpha$ and $\beta$ are skeins such that the residues of their lead terms are complementary, then $Tr(\alpha*\beta)\neq 0$. \end{lemma}

\proof When you expand the product $\alpha*\beta$  as a linear combination of simple diagrams the lead term has residue $\vec{0}$. By Corollary \ref{chi} when we expand $\alpha*\beta$  as a linear combination of threaded primitive diagrams, its lead term is in $\chi(F)$. When we take the trace that term persists since a lead term cannot cancel with anything smaller. \qed

\begin{theorem}[Local Basis Theorem]\label{localbasis} Let $F$ be an oriented surfae of finite type with Euler characteristic $e$. Let $E$ be the set of edges in a given ideal triangulation of $F$. If $\mathcal{B}=\{S_1,\ldots S_{-3e}\}$ is a collection of diagrams on $F$  such that their residues form a basis for $\mathbb{Z}_N^E$, then the set  of skeins 
\begin{equation} 
\mathcal{C} = \{
T_{k_1}(S_1)*T_{k_2}(S_2)*\ldots *T_{k_{-3e}}(S_{-3e})\ | \ 0\leq k_i \leq N-1
\}
\end{equation} is a local basis for $K_N(F)$ over $\chi(F)$.  \end{theorem}

\proof The proof is an application of  Theorem \ref{exhaustbasis}. Recall that the number of edges in an ideal triangulation of a surface of Euler characteristic $e$ is equal to $-3e$.
 We will prove that the set $\mathcal{C}$  exhausts $T$, where $T$ is given by (\ref{tracepairing}). The nondegeneracy of $T$ then implies that the rank of  the matrix $g$, whose entries are all traces of pairwise products of elements of $\mathcal{C}$,  is equal to $N^{|E|}=N^{-3e}$. Since the matrix $g$ has full rank, its determinant is not zero. Since $\chi(F)$ has no zero divisors this implies that the determinant of $g$ is not nilpotent, which completes the proof.

Let $\alpha$ be an arbitrary skein. The residue of its lead term is the tuple $(k_i)\in \{0,\ldots,N-1\}^E$. Its complementary tuple $(l_i)$ satisfies  $(k_i)+(l_i)=\vec{0}$. As the residues of  $\mathcal{B}$ form a basis of $\mathbb{Z}_N^E$ there exists a tuple $b_i\in \{0,\ldots,N-1\}^E$ such that the lead term of \begin{equation} T_{b_1}(S_1)*\ldots *T_{b_{-3e}}(S_{-3e})\end{equation} has its residue equal to $(l_i)$.  Hence the lead term of \begin{equation}\alpha*(T_{b_1}(S_1)*\ldots *T_{b_{-3e}}(S_{-3e}))\end{equation} is divisible by $N$ which means that the lead term of this skein written in terms of the threaded basis is in $\chi(F)$.  This implies that  \begin{equation}T(\alpha\otimes T_{b_1}(S_1)*\ldots T_{b_{-3e}}(S_{-3e}))\neq 0.\end{equation}  Therefore the set $\mathcal{C}$ forms a local basis of $K_N(F)$ over $\chi(F)$ by Theorem \ref{exhaustbasis}. \qed

There is an alternative formulation of this theorem, that has a different hypotheses but the same proof.

\begin{scholium}[Exhaustion]\label{exhaustion} If $\mathcal{B}$ is any collection of skeins such that the residues of the lead terms of the elements of $\mathcal{B}$ are in one to one correspondence with the elements of $\mathcal{Z}_N^E$, then $\mathcal{B}$ is a local basis for $K_N(F)$. \end{scholium}

\proof The set exhausts $T$ and is linearly independent. \qed

Let $S=\chi(F)-\{0\}$. Since $\chi(F)$ is an integral domain, it is multiplicatively closed and we can form $S^{-1}\chi(F)$, the {\bf function field} of the character variety, and $S^{-1}K_N(F)$ which is a vector space over the function field of the character variety.

\begin{cor}\label{dimKoverchi} The dimension of $S^{-1}K_N(F)$ as a vector space over $S^{-1}\chi(F)$ is $N^{-3e(F)}$, where $e(F)$ is the Euler characteristic of surface $F$.\end{cor}

\proof A local basis for $K_N(F)$ over $\chi(F)$ always has $N^{-3e(F)}$ elements as this is the order of $\mathbb{Z}_N^E$. A local basis for $K_N(F)$ induces a basis for $S^{-1}K_N(F)$ because localizing a free module yields a free module. \qed

We conclude the section with some  observations about the structure of the localized algebra.

\begin{theorem}\label{simple} Let $S=\chi(F)-\{0\}$ so that $S^{-1}\chi(F)$ is the function field of the character variety. The algebra $S^{-1}K_N(F)$ has no nontrivial left ideals. \end{theorem}

\proof We show that the only left ideals of $S^{-1}K_N(F)$ are $\{0\}$ and $S^{-1}K_N(F)$. Let $L$ be a left ideal that is not $\{0\}$. Hence it has a  nonzero element $\alpha\in L$. Let $\mathcal{B}$ be any collection of skeins whose lead terms are in one to one correspondence with the elements of $\mathbb{Z}_N^E$, and is hence a basis for $S^{-1}K_N(F)$ by the Scholium \ref{exhaustion}. The skeins $\beta*\alpha$ where $\beta \in \mathcal{B}$ have the property that their lead terms are in one to one correspondence with the elements of $\mathbb{Z}_N^E$, since residues of lead terms are additive under multiplication. However this means that $L$ contains a basis for $S^{-1}K_N(F)$, which in turn implies that $L=S^{-1}K_N(F)$. \qed

\begin{cor} \label{division}The Kauffman bracket skein algebra $K_N(F)$, localized so that every nonzero character is invertible, is a division algebra. \end{cor}

\proof  Let $\alpha$ be a nonzero element of  $ S^{-1}K_N(F)$. The left ideal $S^{-1}K_N(F)\alpha$ is all of $S^{-1}K_N(F)$ by the Theorem \ref{simple}. Therefore there is some skein $\beta\in S^{-1}K_N(F)$ such that $\beta*\alpha=1$. \qed

\section{Pants Decompositions}\label{pants}

The goal of this section is to factor  the Kauffman bracket skein algebra of a punctured surface $F$ as a module over its center. When viewed as a module over the center, an appropriate localization of  $K_N(F)$ can be expressed as a tensor product of two commutative subalgebras generated by pants decompositions of the surface $F$.

Recall that a surface $F$ of genus $g$ with $p$ punctures has Euler characteristic $e=2-2g-p$. Any ideal triangulation of $F$ has $-3e=6g-6+3p$ edges.
Recall also that a {\bf pants decomposition} of a surface $F$ is a collection  of disjoint simple closed curves on $F$ such that every component of the complement of their union is a thrice punctured sphere (i.e. a pair of pants). Note that if $F$ has  genus $g$ and $p$ punctures, then any pants decomposition of $F$ consists of  $3g-3+p$ curves.

  For each puncture $p_i$ choose a simple closed curve $\partial_i$ that bounds a disk containing the puncture $p_i$ and use the same notation for  the skein corresponding to  $\partial_i$ with the blackboard framing. It is easy to see that  $\partial_i\in Z(K_N(F))$ for all $i$. A smallest subring of $K_N(F)$ that  contains $\chi(F)$ and any collection of $\partial_i$  (denoted $\chi(F)[\partial_{i_1},\dots\partial_{i_k}]$) is also contained in the center of $K_N(F)$.
 
 The center $Z(K_N(F))$ is a finite extension of $\chi(F)$.  We characterize it in \cite{FKL} proving the following:

\begin{theorem}[\cite{FKL}]
Given an orientable surface $F$ of negative Euler characteristic and $p\geq 1$ punctures,
the center $Z(K_N(F))$ is generated by $\chi(F)$ and the skeins $\{\partial_1,\dots\partial_p\}$ surrounding the punctures,
\begin{equation}
Z(K_N(F))=\chi(F)[\partial_1,\dots\partial_p].
\end{equation}

\end{theorem}

 Given a pants decomposition $P=\{P_1, \dots , P_{3g-3+p}\}$  of $F$,  let ${\mathcal P}$ denote the subalgebra of $K_N(F)$ generated by the diagrams $\{P_1, \dots , P_{3g-3+p}\}$ with coefficients from $Z(K_N(F))$. This subalgebra is commutative, since the skeins generating it have disjoint diagrams.  
  A slight modification of arguments  in  \cite{AF2} yields:

\begin{theorem}[Extension of scalars]\label{extofsc}
Let $F$ be an orientable surface of genus $g$ with $p$ punctures and negative Euler characteristic, and let ${ P}=\{P_1, \dots , P_{3g-3+p}\}$ be a pants decomposition of $F$.
If ${\rm{Ann}}_i$ denotes a small annulus about the curve $P_i$, then the subalgebra ${\mathcal P}$ generated by the diagrams from $P$ is given as a module over $Z(K_N(F))$ by the following tensor product
\begin{equation}
{\mathcal P} = {\bigotimes_i}_{Z(K_N(F))}\left( K_N({\rm{Ann}}_i)\otimes_{\chi( {\rm{Ann}}_i)} Z(K_N(F))
\right). 
\end{equation}
\end{theorem}
\begin{proof} 

It was shown in  \cite{AF2} that it is possible to find an ordering of the curves $\{P_1, \dots , P_{3g-3+p}\}$  and a collection of lines $\{e_1,\dots, e_{3g-3+p}\}$ through the punctures in such a way that
\begin{equation}\label{uptr}
i(e_i,P_j)= 0 \ \mbox{when} \ \ j>i
\end{equation}
\[i(e_i,P_i) = 1 \ \ \mbox{or}\ \ 2.\]
This collection of lines $\{e_1,\dots, e_{3g-3+p}\}$ can be completed to a collection of $-3e=6g-6+3p$ lines forming a triangulation of $F$. 

 Let $M$ be the $-3e\times -3e$ matrix with elements from $\mathbb{Z}_N$  that has column vectors $res(P_{3g-3+p}),\dots,res(P_2), res(P_1), v_{3g-3+p+1}, \dots,  v_{6g-6+3p}$, where $v_i\in \mathbb{Z}_N^{-3e}$ has $1$ in the $i$-th place and $0$ elsewhere. Conditions (\ref{uptr}) guarantee that $M$ is lower triangular and has determinant equal to a power of $2$. Since $N$ is odd,  $2$ is a unit in $\mathbb{Z}_N$, and this shows that the column vectors of $M$ form a basis of $\mathbb{Z}_N^{-3e}$. This guarantees that the skeins $P_i$ are linearly independent in $K_N(F)$ and proves the theorem.
\end{proof}

\begin{cor}
The subalgebra 
${\mathcal P}$ of $K_N(F)$ generated by a pants decomposition of the surface $F$ is a free module of rank  $N^{3g-3+p}$ over $Z(K_N(F))$. \qed
\end{cor}

Note that since   ${\mathcal P}$  is commutative, the same is true if we localize it over some element of the center of $K_N(F)$.  Freeness is preserved by extending coefficients and by localizing, see \cite{AM}. 
 
In order to check whether a collection of skeins forms a local basis  in constructing a splitting of the Kauffman bracket skein algebra we need an enhancement of the Local Basis Theorem  \ref{localbasis}.

The skeins $\partial_i$ are central, and have degree $N$ over $\chi(F)$. In fact they are independent as a consequence of the computation in \cite{AF2} so that the degree of $\chi(F)[\partial_{i_1},\ldots,\partial_{i_k}]$ over $\chi(F)$ is $N^k$. Hence if the result of adjoining the $\partial_i$ to a set $\mathcal{B}$ is a local basis for $K_N(F)$, then $\mathcal{B}$ is a local basis for $K_N(F)$ over the central subring  $\chi(F)[\partial_{i_1},\ldots,\partial_{i_k}]$, whose rank is $N^{-3e(F)-k}$. 

\begin{theorem}\label{basecase}
 Let $\mathcal{B}$ be any collection of diagrams so that their residues form a basis for $\mathbb{Z}_N^E$  Suppose that some subset of $B$ consists of some of the curves 
 $\partial_i$ bounding the punctures. If  $\mathcal{B}=\{\partial_{i_1},\dots,\partial_{i_k}\}\cup\{S_1,\dots S_{-3e-k}\}$ then  the set of skeins  
 \begin{equation} 
\{
T_{j_1}(S_1)*T_{j_2}(S_2)*\ldots *T_{j_{-3e-k}}(S_{-3e-k})\ | \ 0\leq j_i \leq N-1
\}
\end{equation}
 forms a local basis for $K_N(F)$ over $\chi(F)[\partial_{i_1},\dots \partial_{i_k}]$.
\end{theorem}

\proof  Suppose that $\mathcal{B}=\{\partial_{1_1},\ldots,\partial_{i_k},S_1,\ldots S_{-3e-k}\}$.  By Theorem \ref{localbasis} the skeins 
\begin{equation}\left(\prod_{m=1}^kT_{j_m}(\partial_{i_m})\right)T_{n_1}(S_1)*\dots*T_{n_{-3e-k}}(S_{-3e-k})
\end{equation}  
where the $j_m$ and $n_i$ range from $0$ to $N-1$,  is a basis for the appropriate localization of $K_N(F)$ over the appropriate localization of $\chi(F)$.
This means that the subspace spanned by the $\prod_{m=1}^kT_{j_m}(\partial_{i_m})$ is free of rank $N^{k}$ in the appropriate localization. However this subspace  is just $\chi(F)[\partial_{1_1},\ldots,\partial_{i_k}]$. By the multiplicative property of degree this means that the dimension of the appropriate localization of $K_N(F)$ over the appropriate localization of  $\chi(F)[\partial_{1_1},\ldots,\partial_{i_k}]$ is $N^{-3e(F)-k}$. 
\qed

 \begin{theorem}\label{pantssplit}
 Let $F$ be an orientable surface of genus $g$ with $p$ punctures and negative Euler characteristic, where $p\geq 1$.  There exist two pants decompositions $P$ and $Q$ of $F$ and a skein $c\in Z(K_N(F))$ such that the associated skein modules 
 ${\mathcal P}$ and ${\mathcal Q}$ localized over $S=\{c^k\}$ give a splitting of the localized skein algebra of $F$ over its center.
 \begin{equation}
K_N(F)_c  = \mathcal{P}_c\otimes_{Z(K_N(F))_c}\mathcal{Q}_c.
 \end{equation}
 The intertwiner is given by $p\otimes q \rightarrow p*q$.
\end{theorem}

\proof

We will show how to  choose  two pants decompositions of $F$,  
\begin{equation}P=\{P_1,\dots,P_{3g-3+p}\},  Q =\{Q_1\dots,Q_{3g-3+p}\}
\end{equation} 
so that if $\mathcal{P}$ and $\mathcal{Q}$ are the subalgebras of $K_N(F)$ generated by the curves of $P$ and $Q$ then
\begin{equation} {\mathcal{P}}_c\otimes_{Z({{K}}_N(F))_c}{\mathcal{Q}}_c\cong {K}_N(F)_c.\end{equation}
The skein $c\in Z(K_N(F))$ is equal to the determinant of the matrix $(T(v_i,v_j))$, 
where $v_i$, $v_j$ are from the family of skeins
\begin{equation}\left(\prod_{m=1}^kT_{j_m}(\partial_{i_m})\right)T_{n_1}(P_1)*\dots*T_{n_{3g-3+p}}(P_{3g-3+p})*
T_{n_1}(Q_1)*\dots*T_{n_{3g-3+p}}(Q_{3g-3+p})
\end{equation}  
and $T$ is the trace pairing defined by (\ref{tracepairing}).

The isomorphism is given by sending a tensor product of skeins into their product in $K_N(F)$,
\begin{equation}\label{pantsiso}
p\otimes q \mapsto p *q.
\end{equation}
By theorem \ref{extofsc} the modules $\mathcal{P}$ and $\mathcal{Q}$ are free over 
$Z(K_N(F))$
and 
\begin{equation}
\dim{\mathcal{P}}_{Z({K}_N(F))}= \dim{\mathcal{Q}}_{Z({K}_N(F))}= 3g-3+p.
\end{equation}
The basis of the tensor product  $\mathcal{P}\otimes_{Z({K}_N(F))}\mathcal{Q}$ consists of skeins 
\begin{equation}
\{\prod_iT_{k_i}(P_i)\otimes \prod_jT_{l_j}(Q_j)\  | \  k_i,l_j \in \{0,1,\dots N-1\} \}.
\end{equation}
This tensor product has dimension $N^{6g-6+2p}$ as a module over over  $Z(K_N(F))$. Note that $K_N(F)$ also has dimension $N^{-3e-p}=N^{6g-6+2p}$ over  $Z(K_N(F))$. 

In order to prove that the map given by the equation (\ref{pantsiso}) is an isomorphism we construct a triangulation of the surface $F$ and show that the vectors of the residue classes of  the curves $P\cup Q\cup\{\partial_1,\dots \partial_p\}$ form a basis for ${\mathbb{Z}}_N^E$.  This is achieved by proving that the change of basis matrix from the standard basis of ${\mathbb{Z}}_N^E$ to this basis has determinant which is a unit in ${\mathbb{Z}}_N$. 
 By Theorem \ref{basecase} this implies that the skeins  $\prod_iT_{k_i}(P_i)* \prod_jT_{l_j}(Q_j)$ form a basis for $K_N(F)_c$ over 
$Z({K}_N(F))_c$.

We present an explicit computation of the determinant of the transpose of the change of basis matrix for a genus four surface with one puncture. We  describe the matrix for  arbitrary genus and one puncture and show how to extend the computation to the general case. Finally, we consider what happens for any number of punctures.
 
The first step is to describe the two specific pants decompositions $P$ and $Q$.
In all the descriptions  we order the $g$ handles from top to bottom, where the Morse function is given by the height on the page. 
Diagram $P$ for a surface of genus $3$ with one puncture is shown in Figure \ref{P3}. We depict the puncture by a vertical red segment. There are three curves surrounding the three holes, called $a_1$, $a_2$ and $a_3$, two more nested curves in the front (one surrounding the top two holes, and one surrounding all three holes) denoted $c_1$ and $c_2$, and one curve in the back, drawn with dotted lines, which surrounds the top two holes, denoted $h_1$. Those $6$ curves yield a pants decomposition of a closed surface of genus $3$ (note that $6=3\cdot 3 -3$). Since $F$ has one puncture we need an additional curve which is parallel to one of the other curves with the puncture contained in an annulus co-bounded by these two curves. In  Figure \ref {P3} this is the outermost curve drawn on the back, denoted $h_2$, which co-bounds a punctured  annulus with $c_2$.

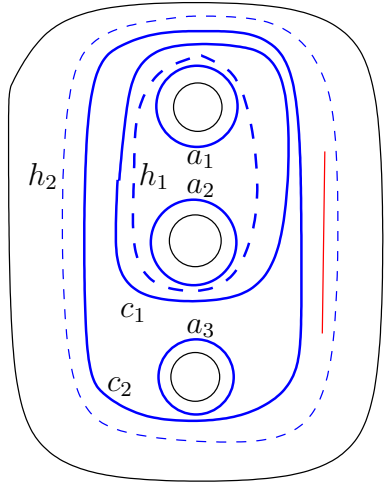
\begin{figure}[H]\begin{center}\begin{picture}(143,182)\includegraphics{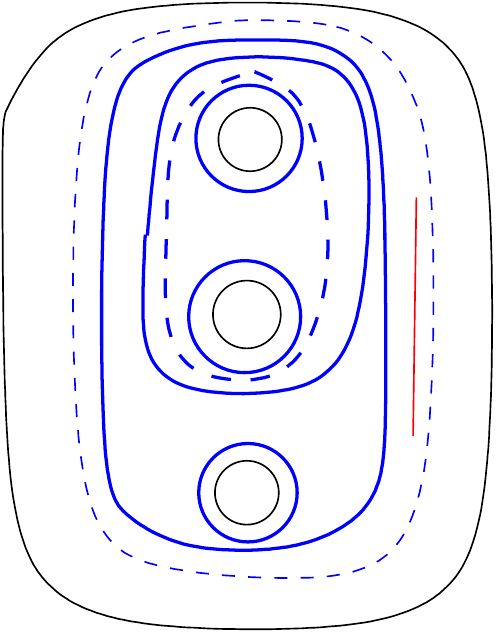}\end{picture}\put(-75,121){$a_1$}
\put(-75,110){$a_2$}
\put(-75,57){$a_3$}
\put(-100,63){$c_1$}
\put(-105,35){$c_2$}
\put(-93,113){$h_1$}
\put(-135,113){$h_2$}
\end{center}
\caption{$P$ for a surface of genus $3$ and $1$ puncture}\label{P3}\end{figure}

When we increase the genus by one we add one more curve surrounding the additional hole, one additional nested curve in the front surrounding all $g$ holes, and one additional  nested curve in the back surrounding $g-1$ holes. Similarly, the outermost curve is doubled and slid across the puncture.

Consequently, for genus g, our preferred pants decomposition $P$ consists of $g$ curves surrounding the $g$ single holes (denoted by $a_1$ through $a_g$), $g-1$ nested curves on the front of the surface surrounding an increasing number of holes from top down, starting with the top $2$ holes and all the way to all $g$ holes, (denoted by $c_1$ through $c_{g-1}$), and the analogous $g-2$ nested curves on the back surrounding from top down $2$, $3$, up to $g-1$ holes (denoted by $h_1$, up to $h_{g-2}$), and finally a curve that is a copy of the curve $c_{g-1}$ surrounding all $g$ holes, slid across the puncture, denoted $h_{g-1}$.

Diagram $Q$ for genus $3$ surface and one puncture, pictured in Figure \ref {Q3}, consists of three curves, each one cutting across each of the three handles (to the left in the figure) which we denote $b_1$, $b_2$ and $b_3$, one additional curve cutting across the middle handle on the other side, denoted $s_1$ (slanted to the right in the figure), two vertical curves cutting through the handles between consecituve holes,  denoted $e_1$ and $e_2$, and a curve that is a copy of the curve $b_3$, slid across the puncture, which we denote by $s_2$.

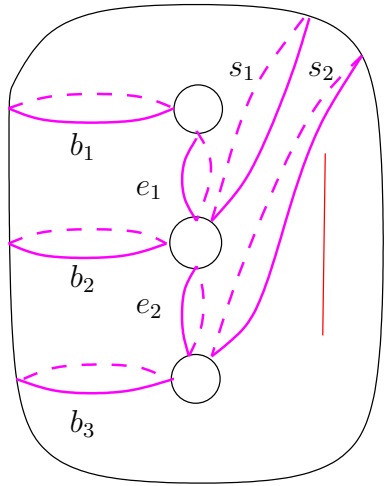
\begin{figure}[H]\begin{center}\begin{picture}(142,183)\includegraphics{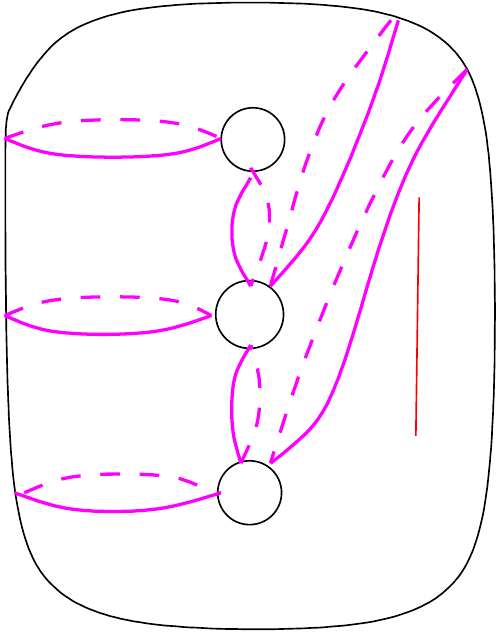}
\put(-120,125){$b_1$}
\put(-120,75){$b_2$}
\put(-120,20){$b_3$}
\put(-95,65){$e_2$}
\put(-95,110){$e_1$}
\put(-60,155){$s_1$}
\put(-30,155){$s_2$}
\end{picture}

\end{center}
\caption{$Q$ for a surface of genus $3$ and $1$ puncture}\label{Q3}\end{figure}

Increasing the genus means adding the two curves cutting through the handles from the additional middle hole to the outside and one cutting the handle joining it to the other hole.

In summary, for genus $g$  our preferred pants decomposition $Q$ consists of $g$ horizontal curves cutting across each handle to the left (denoted $b_1$ through $b_g$),  $g-2$ slanted curves  cutting across the $g-2$ middle handles to their right above the puncture (denoted $s_1$ through $s_{g-2}$), $g-1$ vertical curves cutting  the handles joining adjacent holes (denoted $e_1$ through $e_{g-1}$), and one curve parallel to  $b_g$ slid across the puncture (denoted $s_{g-1}$).

Note that $P$ and $Q$ are chosen in such a way that they intersect minimally.

We will describe additional curves that need to be added to each family when the surface has more than one puncture after we finish the discussion of the once punctured case.

Next step is to describe a preferred triangulation. We do this by slicing the surface horizontally by $g-1$ edges into $g$ handles.  Assume first that the surface has one puncture. All the edges of the triangulation end at the puncture. If we drew the puncture as a single point, the picture would be hard to render. Instead, we draw the puncture as a segment starting at the bottom part of the top handle, running across all the middle handles and ending at the top of the bottom handle. Nevertheless, for each handle we draw the edges meeting at a single point.  The triangulation of the top handle has $5$ edges, four of which triangulate the handle and one is the boundary component denoted $b_1$, see Figure \ref{top}. The middle pieces have $6$ edges each: four of them are translations of the four edges triangulating the top handle, one corresponds to the bottom boundary component, denoted $b_i$, and one additional edge denoted $f_i$ is coming from gluing the handle to the ones above and below it along a triangle, see Figure \ref{middle}. In Figure \ref{middle} the bold vertical arc represents the puncture. The bottom piece has just $4$ edges which are a rotation of the four edges triangulating the top handle, since we are always including the boundary edge in the slice  above.

\begin{figure}[H]\begin{center}\begin{picture}(110,120)\scalebox{1}{\includegraphics{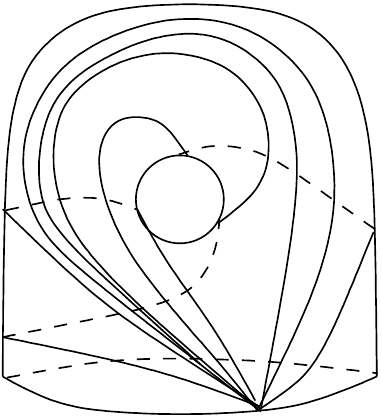}}\put(-120,20){$1$}\put(-120,55){$2$}\put(-75,75){$3$}\put(0,50){$4$}\put(0,5){$b_1$}\end{picture}\end{center}
\caption{Triangulation of the top handle}\label{top}\end{figure}
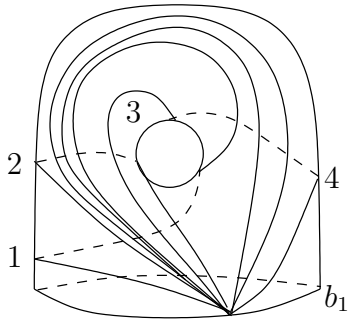

\begin{figure}[H]\begin{center}\begin{picture}(116,131)\scalebox{1}{\includegraphics{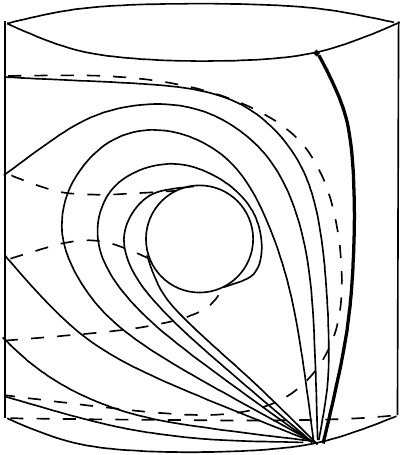}}\put(-130,15){$f_i$}\put(-130,30){$1$}\put(-130,52){$2$}\put(-105,62){$3$}\put(-130,75){$4$}\put(0,5){$b_i$}\end{picture}\end{center}
\caption{Triangulation of a middle handle}\label{middle}\end{figure}
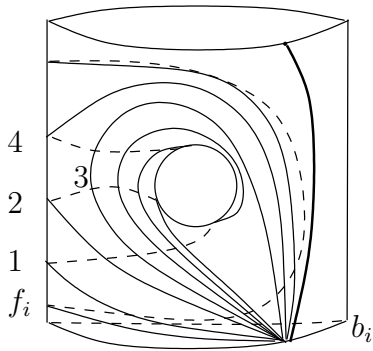

\begin{figure}[H]\begin{center}\begin{picture}(110,120)\scalebox{1}{\includegraphics{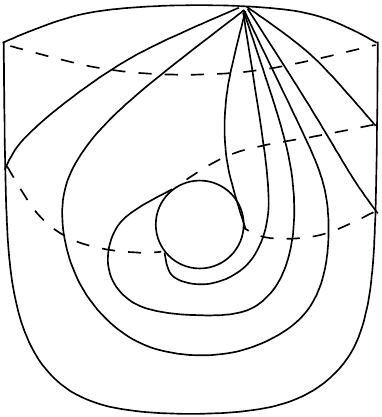}}\put(0,80){$1$}\put(0,55){$2$}\put(-22,20){$3$}\put(-120,68){$4$}\end{picture}\end{center}
\caption{Triangulation of the bottom handle}\label{bottom}\end{figure}
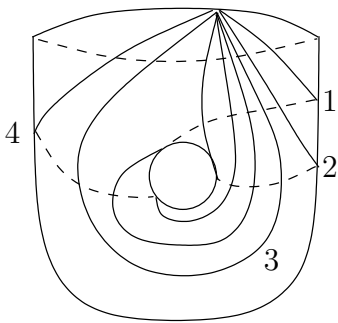

We will describe the additional pieces of triangulation that are needed when we have more than one puncture after we finish the proof for the case of a surface with one puncture.

Recall that $E$ denotes the set of edges of the triangulation. We  describe the vectors in 
${\mathbb Z}^E$ that give the admissible colorings  $f_{P_i}$ and $f_{Q_j}$  corresponding to the curves $P_i$ and $Q_j$ in the triangulation chosen above.
Recall that the coordinates of those vectors are geometric intersection numbers of the curves with the edges of the triangulation. Although Theorem \ref{basecase} requires that we look at the residue classes mod $N$, it turns out that it suffices to work with integers, since the determinant of the change of basis matrix is a power of $2$, which is a unit in ${\mathbb Z}_N$ for any odd $N$. 

The coordinates of ${\mathbb Z}^E$ are ordered in the following way. Imagine as before that the height on the page is a Morse function and the $g$ handles are ordered from top to bottom. The first five coordinates are the intersection numbers with the edges $\{1,2,3,4, b_1\}$ from the top handle as  indicated in Figure \ref{top}. The following $6(g-2)$ coordinates come from intersections with the edges triangulating the consecutive $g-2$ middle handles ordered from top to bottom, where each $6$-tuple consists of intersection numbers with edges  
$\{f_i,1,2,3,4, b_i\}$ as pictured in Figure \ref{middle}. The last four coordinates are the intersection numbers with the edges $\{1,2,3,4\}$ triangulating the bottom handle,  as shown in Figure \ref{bottom}.

We abbreviate vectors of length $4$ that occur repeatedly as consecutive coordinates of curves in our two systems $P$ and $Q$ in the following way.

\begin{equation}\label{lexicon}
\begin{split}
\vec{a}= (1,1,0,1), \vec{b} = (1,0,1,1), \vec{c}=(1,1,2,3), \vec{d}=(3,3,2,1)\\
\vec{e}=(3,2,1,1), \vec{f}=(3,3,2,3), \vec{2}=(2,2,2,2), \vec{0}=(0,0,0,0)
\end{split}
\end{equation}

These vectors satisfy some relations that we will use repeatedly.
\begin{equation}\label{edfrelations}
\begin{split}
\vec{d} =  \vec{e}+\vec{2}-\vec{a}-\vec{b}\\
\vec{e}  =  \vec{2}+\vec{a}+\vec{b}-\vec{c}\\
\vec{f}=\vec{a}+\vec{2}
\end{split}
\end{equation}

The diagrams in Figures \ref{factotum1}, \ref{factotum2} and \ref{factotum3} encode the curves in the families $P$ and $Q$ as projections onto the coordinates corresponding to their intersections with the handles.
\begin{figure}[H]\begin{center}\begin{tabular}{ccc} \scalebox{.67}{\includegraphics{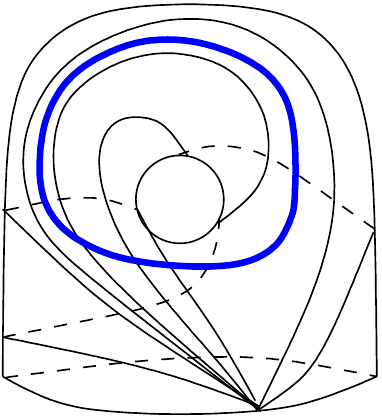}} & \scalebox{.67}{\includegraphics{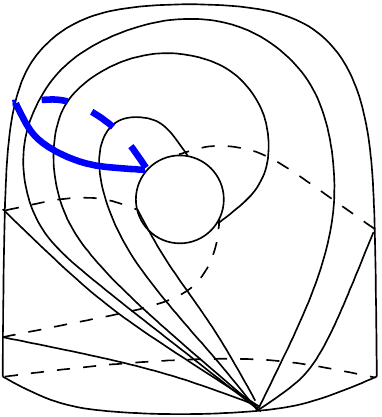}} & \scalebox{.67}{\includegraphics{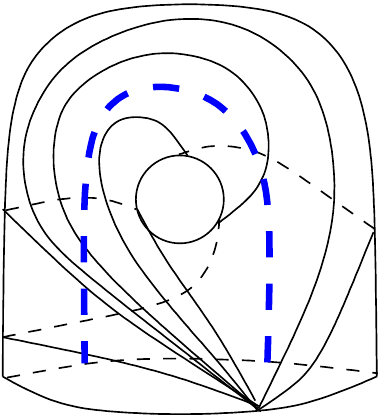}}\\ $11010=\vec{a}0$ & $10110=\vec{b}0$ & $11012=\vec{a}2$ \\ \scalebox{.67}{\includegraphics{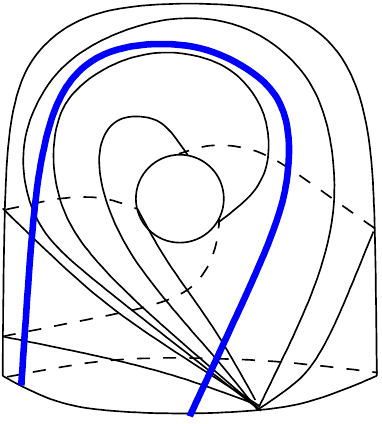}} &\scalebox{.67}{ \includegraphics{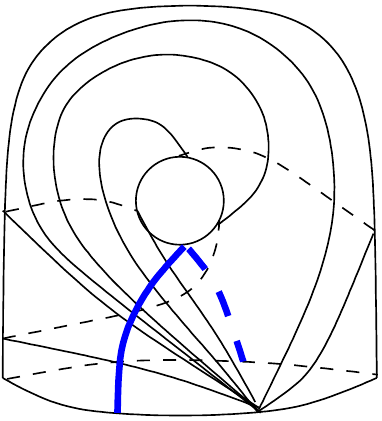}} &\scalebox{.67}{ \includegraphics{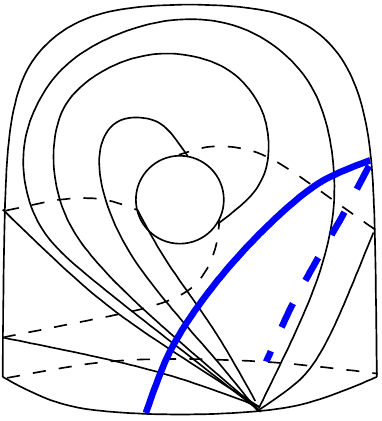}}\\ $33212=\vec{d}2$ & $32112=\vec{e}2$ &$22222=\vec{2}2$ \end{tabular}\caption{Vectors for curves in the top handle}\label{factotum1}\end{center} \end{figure}

\begin{figure}[H]\begin{center}\begin{tabular}{ccc} \scalebox{.67}{\includegraphics{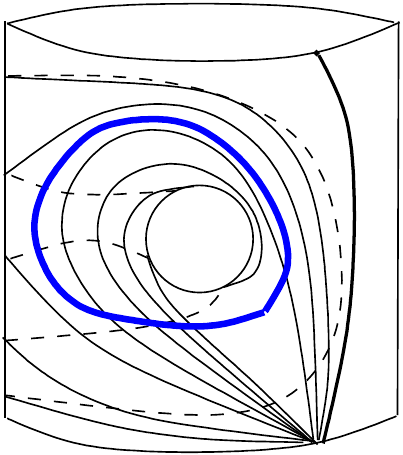}} & \scalebox{.67}{\includegraphics{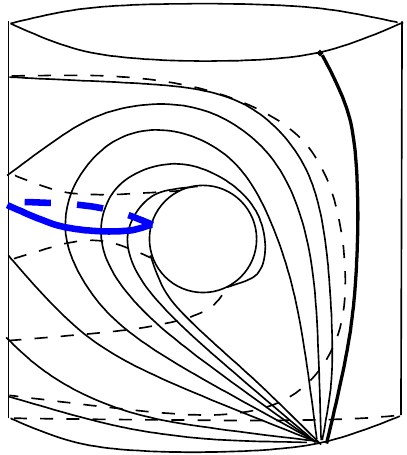}} & \scalebox{.67}{\includegraphics{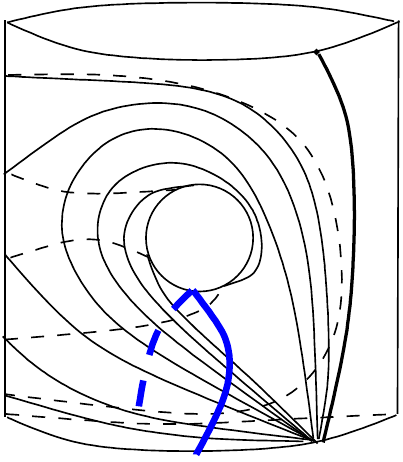}}\\ $011010=0\vec{a}0$ & $010110=0\vec{b}0$ & $232112=2\vec{e}2$ \\ \scalebox{.67}{\includegraphics{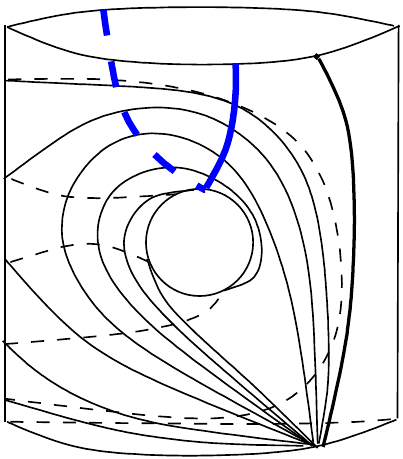}} &\scalebox{.67}{ \includegraphics{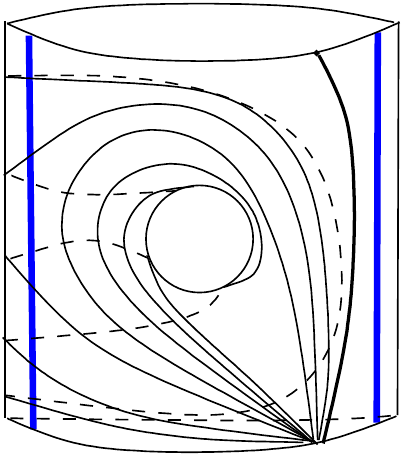}} &\scalebox{.67}{ \includegraphics{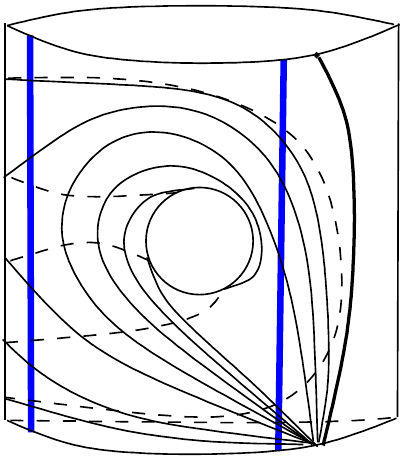}}\\ $210110=2\vec{b}0$ & $211012=2\vec{a}2$ &$433232=4\vec{f}2$\\ \scalebox{.67}{\includegraphics{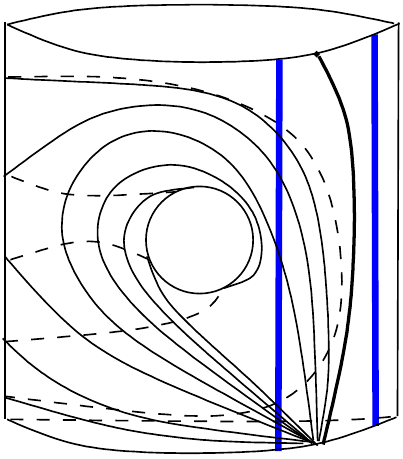}} &\scalebox{.67}{ \includegraphics{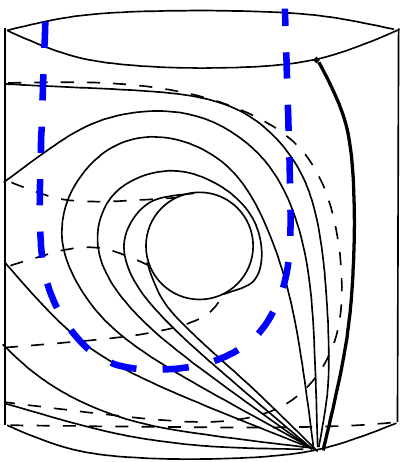}} &\scalebox{.67}{ \includegraphics{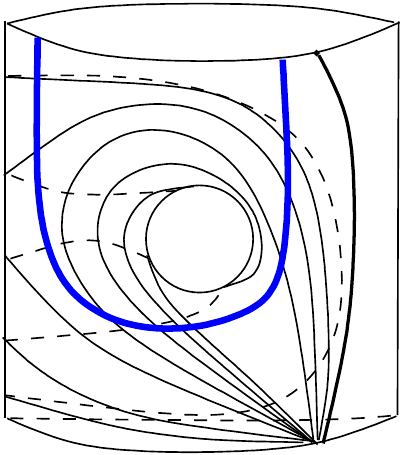}}\\ $222222=2\vec{2}2$ & $211010=2\vec{a}0$ &$211230=2\vec{c}0$\\ \end{tabular}\caption{Vectors for curves in the middle handles}\label{factotum2}\end{center} \end{figure}

\begin{figure}[H]\begin{center}\begin{tabular}{ccc} \scalebox{.67}{\includegraphics{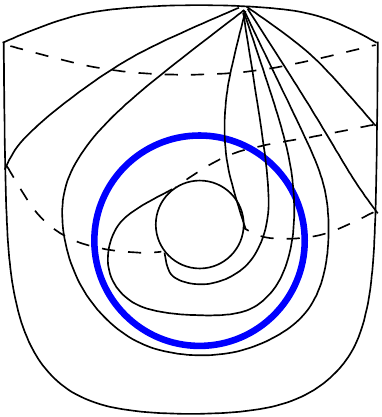}} & \scalebox{.67}{\includegraphics{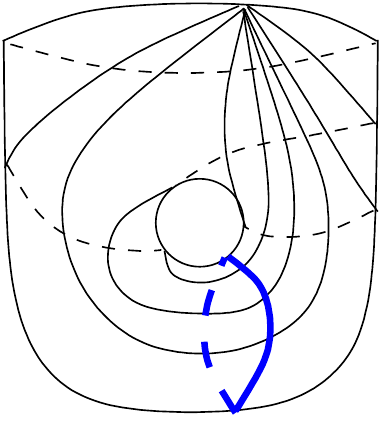}} & \scalebox{.67}{\includegraphics{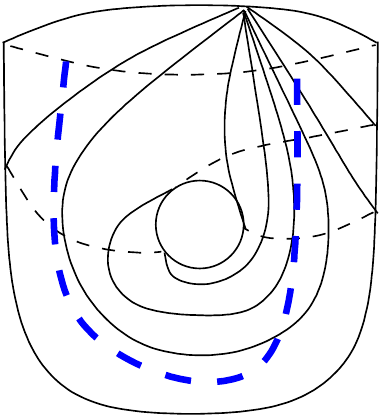}}\\ $1101=\vec{a}$ & $1011=\vec{b}$ & $1101=\vec{a}$ \\ \ &\scalebox{.67}{ \includegraphics{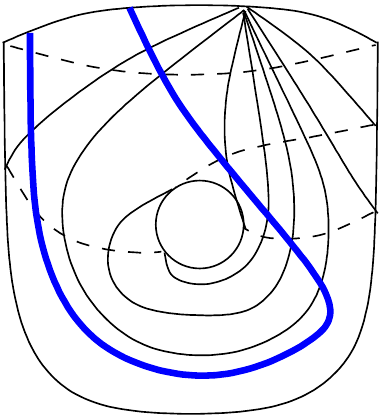}} &\scalebox{.67}{ \includegraphics{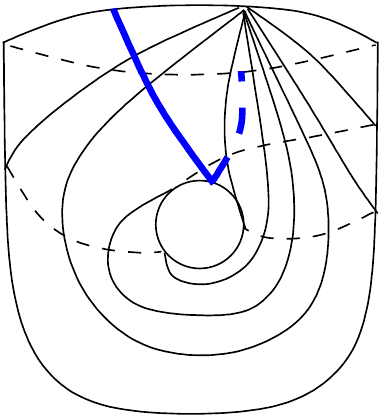}}\\ \ & $1123=\vec{c}$ &$1011=\vec{b}$ \end{tabular}\caption{Vectors for curves in the bottom handle}\label{factotum3}\end{center} \end{figure}

We follow the conventions indicated in Figures \ref{P3} and \ref{Q3} for naming the curves in families $P$ and $Q$ and we order the collection of the curves $P\cup Q$ in $g$ consecutive blocks corresponding to $g$ consecutive handles from top to bottom as follows. The first block corresponding to the top handle consists of  $5$ curves: $a_1$, $b_1$,  $e_1$, $s_1$ and $h_1$. The blocks corresponding to the  middle handles consist of $6$ curves each: $a_i$, $b_i$, $e_i$, $s_i$, $h_i$ and $c_{i-1}$. The block corresponding to the bottom handle has  three curves from $P\cup Q$ and the curve surrounding the puncture: $a_g$, $b_g$, $c_{g-1}$ and $\partial_1$.

We have done explicit computation of the determinant for surfaces of genus $2$ and $3$. Although the matrices are smaller and the computation is easier, one does not see the full pattern of the computation when the genus is increased. Consequently we start by computing the determinant of the transpose of the change of basis matrix for  vectors $f_{P_i}$ and $f_{Q_j}$  for surface of genus $4$ and one puncture, see (\ref{matr}).  Since we are only interested in the determinant, we can proceed using elementary row operations.

Though we are only showing the matrix for genus 4 with one puncture we give a procedure to compute the determinant for any genus. The size of the matrix for a surface of genus $g$ is $(6g-3)\times (6g-3)$. The idea is that after performing some elementary row operations the matrix will be in block  form. The modified matrix starting in the upper left hand corner consists of  a $(5\times 5)$-block, followed by $g-2$ blocks of size $6\times 6$, and finally a $(4\times 4)$-block. Keep in mind that the vectors in the columns labeled $\vec{h}_i$ are of length $4$. Explicit computation yields that the determinant of the  $(5\times 5)$-block is $-2^3$, the determinants of each of the subsequent $g-2$ blocks are equal to  $2^4$ and the determinant of the last block is $2^2$. Thus the determinant of the matrix for a surface of genus $g$ with one puncture equals $-2^{4g-3}$.

The rows of the matrix (\ref{matr}) are labeled by the names of the curves and the entries in the rows are the coordinates of those curves in ${\mathbb{Z}}^E$ ordered as described above and indicated by labels. We label the columns corresponding to the triangulation of the $i$-th handle by 
$\vec{h}_i$ and use the vectors described in (\ref{lexicon}) to denote the coordinates of the curves coming from those handles.
We note that although the notation $h_i$ was also used for some curves in family $P$ it is easy to tell the two objects apart, since the later symbols do not have arrows.

\begin{equation}\label{matr}
\begin{array}{l|cc|ccc|ccc|c}
\ & \vec{h}_1&b_1&f_2 & \vec{h}_2& b_2&f_3&\vec{h}_3&b_3&\vec{h}_4\\
\hline
a_1 &\vec{a}& 0 & 0 & \vec{0} & 0&0&\vec{0}&0&\vec{0} \\
b_1&\vec{b}& 0 & 0 & \vec{0} & 0&0&\vec{0}&0&\vec{0} \\
e_1&\vec{e}&2&2&\vec{b}& 0&0&\vec{0}&0&\vec{0} \\
s_1&\vec{2}&2&2&\vec{b}& 0&0&\vec{0}&0&\vec{0} \\
h_1&\vec{a}&2&2&\vec{a}& 0&0&\vec{0}&0&\vec{0} \\
\hline
a_2&\vec{0}&0&0&\vec{a}& 0&0&\vec{0}&0&\vec{0} \\
b_2&\vec{0}&0&0&\vec{b}& 0&0&\vec{0}&0&\vec{0} \\
e_2&\vec{0}&0&2&\vec{e}&2&2&\vec{b}&0&\vec{0}\\
s_2&\vec{2}&2&2&\vec{2}&2&2&\vec{b}&0&\vec{0}\\
h_2&\vec{a}&2&2&\vec{a}&2&2&\vec{a}&0&\vec{0}\\
c_1&\vec{d}&2&2&\vec{c}& 0&0&\vec{0}&0&\vec{0} \\
\hline
a_3 &\vec{0}& 0 & 0 & \vec{0} & 0&0&\vec{a}&0&\vec{0} \\
b_3 &\vec{0}& 0 & 0 & \vec{0} & 0&0&\vec{b}&0&\vec{0} \\
e_3&\vec{0}&0&0&\vec{0}&0&2&\vec{e}&2&\vec{b}\\
s_3&\vec{2}&2&2&\vec{2}&2&2&\vec{2}&2&\vec{b}\\
h_3&\vec{a}&2&2&\vec{a}&2&2&\vec{a}&2&\vec{a}\\
c_2&\vec{d}&2&4&\vec{f}& 2&2&\vec{c}&0&\vec{0} \\
\hline
a_4 &\vec{0}& 0 & 0 & \vec{0} & 0&0&\vec{0}&0&\vec{a} \\
b_4&\vec{0}& 0 & 0 & \vec{0} & 0&0&\vec{0}&0&\vec{b} \\
c_3&\vec{d}&2&4&\vec{f}& 2&4&\vec{f}&2&\vec{c} \\
\partial_1&\vec{2}&2&2&\vec{2}& 2&2&\vec{2}&2&\vec{2} \\
\end{array}\end{equation}

For a surface with genus $g>4$ and one puncture the matrix  is similarly divided into blocks and columns corresponding to handles from top down.
The entries in the rows $a_1$, $b_1$, $e_1$, $s_1$ and $h_1$  for the top handle are the same as in  the matrix (\ref{matr}) in the first $21$ columns, with zeroes  in the additional columns. The same is true for the rows corresponding to the top middle handle.

In general, in all but the last  middle handle the row vectors of the matrix are as follows, with the index $i$ ranging from $2$ to $g-2$:
row vectors  $a_i$ and $b_i$ have $\vec{a}$ and $\vec{b}$ respectively in the column labeled $\vec{h}_i$ and zeroes in every other column; 
row $e_i$ has zeroes  in columns $\vec{h}_1$ through $b_{i-1}$ followed by $(2, \vec{e}, 2,2,\vec{b})$ in columns $f_i$, $\vec{h}_i$,$b_i$, $f_{i+1}$,$\vec{h}_{i+1}$ and zeroes in the remaining columns;
 row  $s_i$ has twos  in columns $h_1$ through $f_i$, $\vec{b}$ in the column  labeled $\vec{h}_{i+1}$ and zeroes in the remaining columns; 
row  $h_i$ has $(\vec{a},2)$ in columns $(\vec{h}_1,b_1)$, followed by $(2,\vec{a},2)$ in columns $(f_j,\vec{h}_j,b_j)$ for $ j=2,\dots, i$, then $(2,\vec{a})$ in $(f_{i+1},\vec{h}_{i+1})$, and zeroes in the remaining columns;
 row $c_i$ has $(\vec{d}, 2)$ in columns $(\vec{h}_1,b_1)$, then $(4,\vec{f},2)$ in columns $(f_j,\vec{h}_j,b_j)$ for $j=2,\dots i$, followed by $(2,\vec{c},0)$ in columns  $(f_{i+1}, \vec{h}_{i+1}, b_{i+1})$ followed by zeroes in the remaining columns.
 
 The last middle handle has the rows  $a_{g-1}$, $b_{g-1}$ as in all the other handles, with $\vec{a}$ and respectively $\vec{b}$ in the $\vec{h}_{g-1}$  column and zeroes elsewhere.
 The row  $e_{g-1}$ has zeroes  in columns $\vec{h}_1$ through $b_{g-2}$ followed by 
 $(2, \vec{e}, 2)$  in columns $(f_{g-1}, \vec{h}_{g-1} b_{g-1})$ and $\vec{b}$ in column $\vec{h}_g$. 
 The row  $s_{g-1}$  has $2$ in all columns preceding the column $\vec{h}_g$ where the entry is 
 $\vec{b}$.
 The row  $h_{g-1}$ has $(\vec{a},2)$ in columns $(h_1,b_1)$, followed by $(2,\vec{a},2)$ in columns $(f_j,\vec{h}_j,b_j)$ for $ j=2,\dots, g-1$,  and $\vec{a}$ in the $\vec{h}_g$ column.
 The row  $c_{g-2}$ is the same as in the preceding middle handles, with  $(\vec{d}, 2)$ in columns 
 $(\vec{h}_1,b_1)$, then $(4,\vec{f},2)$ in columns $(f_j,\vec{h}_j,b_j)$ for $j=2,\dots g-2$, and $(2,\vec{c},0,\vec{0})$  in the columns  $(f_{g-1}, \vec{h}_{g-1}, b_{g-1}, \vec{h}_g)$.
 
 Finally, in the bottom handle,  rows $a_g$ and $b_g$ have zeroes in every column but $\vec{h}_g$ where they have $\vec{a}$ and $\vec{b}$ respectively; 
 the entries in row $c_{g-1}$ are   $(\vec{d}, 2)$ in columns 
 $(\vec{h}_1,b_1)$, then $(4,\vec{f},2)$ in columns $(f_j,\vec{h}_j,b_j)$ for $j=2,\dots g-1$, and $\vec{c}$ in column $\vec{h}_g$;
all entries in the last row labeled $\partial_1$ are $2$.

When we perform the elementary row operations that yield a matrix in block  form we work consecutively with blocks of the matrix corresponding to the handles. This allows us to see how to generalize this procedure for a surface of arbitrary genus.

The rows $a_i$ and $b_i$  of the matrix (\ref{matr}) consist of zeroes except for the columns corresponding to the handle that the curves live in.  We use them to clear out occurrences of $\vec{a}$ and $\vec{b}$ in those columns in all other rows. In order to avoid ambiguity in naming, once we have altered a row, we refer to it by its row number. For instance the row labeled $s_1$ in the first matrix becomes $r_4$ in the second matrix.

\begin{equation}\label{matr1}
\begin{array}{l|cc|ccc|ccc|c}
\ & \vec{h}_1&b_1&f_2 & \vec{h}_2& b_2&f_3&\vec{h}_3&b_3&\vec{h}_4\\
\hline
a_1 &\vec{a}& 0 & 0 & \vec{0} & 0&0&\vec{0}&0&\vec{0} \\
b_1&\vec{b}& 0 & 0 & \vec{0} & 0&0&\vec{0}&0&\vec{0} \\
r_3&\vec{e}&2&2&\vec{0}& 0&0&\vec{0}&0&\vec{0} \\
r_4&\vec{2}&2&2&\vec{0}& 0&0&\vec{0}&0&\vec{0} \\
r_5&\vec{0}&2&2&\vec{0}& 0&0&\vec{0}&0&\vec{0} \\
\hline
a_2&\vec{0}&0&0&\vec{a}& 0&0&\vec{0}&0&\vec{0} \\
b_2&\vec{0}&0&0&\vec{b}& 0&0&\vec{0}&0&\vec{0} \\
r_8&\vec{0}&0&2&\vec{e}&2&2&\vec{0}&0&\vec{0}\\
r_9&\vec{2}&2&2&\vec{2}&2&2&\vec{0}&0&\vec{0}\\
r_{10}&\vec{0}&2&2&\vec{0}&2&2&\vec{0}&0&\vec{0}\\
c_1&\vec{d}&2&2&\vec{c}& 0&0&\vec{0}&0&\vec{0} \\
\hline
a_3 &\vec{0}& 0 & 0 & \vec{0} & 0&0&\vec{a}&0&\vec{0} \\
b_3 &\vec{0}& 0 & 0 & \vec{0} & 0&0&\vec{b}&0&\vec{0} \\
r_{14}&\vec{0}&0&0&\vec{0}&0&2&\vec{e}&2&\vec{0}\\
r_{15}&\vec{2}&2&2&\vec{2}&2&2&\vec{2}&2&\vec{0}\\
r_{16}&\vec{0}&2&2&\vec{0}&2&2&\vec{0}&2&\vec{0}\\
c_2&\vec{d}&2&4&\vec{f}& 2&2&\vec{c}&0&\vec{0} \\
\hline
a_4 &\vec{0}& 0 & 0 & \vec{0} & 0&0&\vec{0}&0&\vec{a} \\
b_4&\vec{0}& 0 & 0 & \vec{0} & 0&0&\vec{0}&0&\vec{b} \\
c_3&\vec{d}&2&4&\vec{f}& 2&4&\vec{f}&2&\vec{c} \\
\partial_1&\vec{2}&2&2&\vec{2}& 2&2&\vec{2}&2&\vec{2} \\
\end{array}\end{equation}

 The only entries that keep the matrix (\ref{matr1}) from being lower-triangular are collections of three $2$'s in the  columns labeled $f_i$  above the blocks corresponding to the $i$-th handles.  We proceed inductively on the index of the $f_i$. Hence the first step is to clear the $2$'s in the column $f_2$ that are in rows three through five.  To this end, we focus on the minor consisting of rows $1$ through $11$ and columns $1$ through $12$,  noting that the entries in the columns $13$ and above in those rows are all $0$.

\begin{equation}\label{matr2}
\begin{array}{l|cc|ccc|c}
  \ & \vec{h}_1&b_1&f_2 & \vec{h}_2& b_2&f_3\\
\hline
a_1 &\vec{a}& 0 & 0 & \vec{0} & 0&0\\
b_1&\vec{b}& 0 & 0 & \vec{0} & 0&0\\
r_3&\vec{e}&2&2&\vec{0}& 0&0 \\
r_4&\vec{2}&2&2&\vec{0}& 0&0 \\
r_5&\vec{0}&2&2&\vec{0}& 0&0 \\
\hline
a_2&\vec{0}&0&0&\vec{a}& 0&0 \\
b_2&\vec{0}&0&0&\vec{b}& 0&0 \\
r_8&\vec{0}&0&2&\vec{e}&2&2\\
r_9&\vec{2}&2&2& \vec{2}&2&2\\
r_{10}&\vec{0}&2&2&\vec{0}&2&2\\
c_1&\vec{d}&2&2&\vec{c}& 0&0 \\
\end{array}\end{equation}

First we subtract row $r_4$ from row $r_9$,  subtract row $r_5$ from $r_{10}$, and subtract  the combination  $r_3+r_4-r_5-a_1-b_1$ from row $c_1$. Here we are using the relationship $\vec{d}=\vec{e}+\vec{2}-\vec{a}-\vec{b}$ (see (\ref{edfrelations})). Our goal is to have zeroes in the $\vec{h}_1$ and $b_1$ columns in all rows $a_2$ through $r_{11}$.

\begin{equation}\label{matr3}
\begin{array}{l|cc|ccc|c}
  \ & \vec{h}_1&b_1&f_2 & \vec{h}_2& b_2&f_3\\
\hline
a_1 &\vec{a}& 0 & 0 & \vec{0} & 0&0\\
b_1&\vec{b}& 0 & 0 & \vec{0} & 0&0\\
r_3&\vec{e}&2&2&\vec{0}& 0&0 \\
r_4&\vec{2}&2&2&\vec{0}& 0&0 \\
r_5&\vec{0}&2&2&\vec{0}& 0&0 \\
\hline
a_2&\vec{0}&0&0&\vec{a}& 0&0 \\
b_2&\vec{0}&0&0&\vec{b}& 0&0 \\
r_8&\vec{0}&0&2&\vec{e}&2&2\\
r_9&\vec{0}&0&0&\vec{2}&2&2\\
r_{10}&\vec{0}&0&0&\vec{0}&2&2\\
r_{11}&\vec{0}&0&0&\vec{c}& 0&0 \\
\end{array}\end{equation}

 Using the fact that $\vec{e}-\vec{2}-\vec{a}-\vec{b}+\vec{c}=\vec{0}$, we have
\begin{equation} r_8-a_2-b_2-r_9+r_{11}=(\vec{0}02\vec{0}00)\end{equation}

We can subtract this quantity from rows $r_3$, $r_4$ and $r_5$ of matrix (\ref{matr3}) to get :

\begin{equation}\label{matr4}
\begin{array}{l|cc|ccc|c}
  \ & \vec{h}_1&b_1&f_2 & \vec{h}_2& b_2&f_3\\
\hline
a_1 &\vec{a}& 0 & 0 & \vec{0} & 0&0\\
b_1&\vec{b}& 0 & 0 & \vec{0} & 0&0\\
r_3&\vec{e}&2&0&\vec{0}& 0&0 \\
r_4&\vec{2}&2&0&\vec{0}& 0&0 \\
r_5&\vec{0}&2&0&\vec{0}& 0&0 \\
\hline
a_2&\vec{0}&0&0&\vec{a}& 0&0 \\
b_2&\vec{0}&0&0&\vec{b}& 0&0 \\
r_8&\vec{0}&0&2&\vec{e}&2&2\\
r_9&\vec{0}&0&0&\vec{2}&2&2\\
r_{10}&\vec{0}&0&0&\vec{0}&2&2\\
r_{11}&\vec{0}&0&0&\vec{c}& 0&0 \\
\end{array}\end{equation}

Notice that the determinant of the $5\times5$ block in the upper left hand corner is $-2^3$. Since we are working modulo $N$ where $N$ is odd, this is a unit, and the first five rows form a basis for the subspace of ${\mathbb Z}^E_N$ that only has non-zero entries in the first $5$ coordinates. 
 We can add combinations of the first five rows to clear out the non-zero entries in first five columns of all the lower rows of the whole matrix. This means that in the next part of the argument we can concentrate solely  on the entries in columns $6$ and up. 
   
Consider the block of the modified matrix starting at row $6$ and column 6 and going across to column 17 and down to row 17.  In our example, the $18$-th through $21$-st columns (labeled as $\vec{h}_4$) are all zeroes in this range. However in the case that the genus of the surface is greater than $4$ the rows $r_{14}$, $r_{15}$ and $r_{16}$ will have $2$'s in the $18$-th column.  This does not  change anything since 
the linear combinations we take cancel in the $18$-th column.
\begin{equation}\label{matr5}
\begin{array}{l|ccc|ccc}
\ &f_2&\vec{h}_2&b_2&f_3&\vec{h_3}&b_3 \\
\hline
 a_2&0&\vec{a}& 0&0&\vec{0}&0\\
b_2&0&\vec{b}& 0&0&\vec{0}&0 \\
r_8&2&\vec{e}&2&2&\vec{0}&0\\
r_9&0&\vec{2}&2&2&\vec{0}&0\\
r_{10}&0&\vec{0}&2&2&\vec{0}&0\\
r_{11}&0&\vec{c}& 0&0&\vec{0}&0\\
\hline
a_3 & 0 & \vec{0} & 0&0&\vec{a}&0 \\
b_3 & 0 & \vec{0} & 0&0&\vec{b}&0 \\
r_{14}&0&\vec{0}&0&2&\vec{e}&2\\
r_{15}&2&\vec{2}&2&2&\vec{2}&2\\
r_{16}&2&\vec{0}&2&2&\vec{0}&2\\
c_2&4&\vec{f}& 2&2&\vec{c}&0 \\
\end{array}\end{equation}

Our goal is to use row operations to have this matrix equal to the result of appending a column that has all but one entry equal to $0$ to the left hand side of the matrix (\ref{matr4}), thus creating an iterative procedure.

The first step is to subtract $r_9$ from $r_{15}$, and $r_{10}$ from $r_{16}$.
This echoes  subtracting  row $r_4$ from row $r_9$, and row $r_5$ from $r_{10}$  in the matrix (\ref{matr2}).

\begin{equation}\label{matr6}
\begin{array}{l|ccc|ccc}
\ &f_2&\vec{h}_2&b_2&f_3&\vec{h_3}&b_3 \\
\hline
 a_2&0&\vec{a}& 0&0&\vec{0}&0\\
b_2&0&\vec{b}& 0&0&\vec{0}&0 \\
r_8&2&\vec{e}&2&2&\vec{0}&0\\
r_9&0&\vec{2}&2&2&\vec{0}&0\\
r_{10}&0&\vec{0}&2&2&\vec{0}&0\\
r_{11}&0&\vec{c}& 0&0&\vec{0}&0\\
\hline
a_3 & 0 & \vec{0} & 0&0&\vec{a}&0 \\
b_3 & 0 & \vec{0} & 0&0&\vec{b}&0 \\
r_{14}&0&\vec{0}&0&2&\vec{e}&2\\
r_{15}&2&\vec{0}&0&0&\vec{2}&2\\
r_{16}&2&\vec{0}&0&0&\vec{0}&2\\
c_2&4&\vec{f}& 2&2&\vec{c}&0 \\
\end{array}\end{equation}

Next we replace row $r_{15}$ by $r_{15}-r_8+r_9+a_2+b_2-r_{11}$, row $r_{16}$ by $r_{16}-r_8+r_9+a_2+b_2-r_{11}$ and row $c_2$ by $c_2+a_2+2b_2-2r_8+r_9-2r_{11}$.  Using the identities $\vec{e}-\vec{2}-\vec{a}-\vec{b}+\vec{c}=\vec{0}$, $\vec{f}-\vec{e}+\vec{b}-\vec{c}=\vec{0}$,  and $\vec{f}=\vec{a}+\vec{2}$ we get:

\begin{equation}\label{matr8}
\begin{array}{l|ccc|ccc}
\ &f_2&\vec{h}_2&b_2&f_3&\vec{h_3}&b_3 \\
\hline
 a_2&0&\vec{a}& 0&0&\vec{0}&0\\
b_2&0&\vec{b}& 0&0&\vec{0}&0 \\
r_8&2&\vec{e}&2&2&\vec{0}&0\\
r_9&0&\vec{2}&2&2&\vec{0}&0\\
r_{10}&0&\vec{0}&2&2&\vec{0}&0\\
r_{11}&0&\vec{c}& 0&0&\vec{0}&0\\
\hline
a_3 & 0 & \vec{0} & 0&0&\vec{a}&0 \\
b_3 & 0 & \vec{0} & 0&0&\vec{b}&0 \\
r_{14}&0&\vec{0}&0&2&\vec{e}&2\\
r_{15}&0&\vec{0}&0&0&\vec{2}&2\\
r_{16}&0&\vec{0}&0&0&\vec{0}&2\\
r_{17}&0&\vec{0}& 0&0&\vec{c}&0 \\
\end{array}\end{equation}
Note that the column $f_2$ has $2$ in row $r_8$ and zeroes in all other rows.
The matrix (\ref{matr8}) coincides with appending this column on the left to matrix (\ref{matr4}) as desired. We now proceed as before.  

Putting all the pieces together, the block form of the matrix obtained by elementary row operations from matrix (\ref{matr}) is

\begin{equation}\label{blockmatrix}
\begin{array}{l|cc|ccc|ccc|c}
\ & \vec{h}_1&b_1&f_2 & \vec{h}_2& b_2&f_3&\vec{h}_3&b_3&\vec{h}_4\\
\hline
a_1 &\vec{a}& 0 & 0 & \vec{0} & 0&0&\vec{0}&0&\vec{0} \\
b_1&\vec{b}& 0 & 0 & \vec{0} & 0&0&\vec{0}&0&\vec{0} \\
r_3&\vec{e}&2&0&\vec{0}& 0&0&\vec{0}&0&\vec{0} \\
r_4&\vec{2}&2&0&\vec{0}& 0&0&\vec{0}&0&\vec{0} \\
r_5&\vec{0}&2&0&\vec{0}& 0&0&\vec{0}&0&\vec{0} \\
\hline
a_2&\vec{0}&0&0&\vec{a}& 0&0&\vec{0}&0&\vec{0} \\
b_2&\vec{0}&0&0&\vec{b}& 0&0&\vec{0}&0&\vec{0} \\
r_8&\vec{0}&0&2&\vec{e}&2&0&\vec{0}&0&\vec{0}\\
r_9&\vec{0}&0&0&\vec{2}&2&0&\vec{0}&0&\vec{0}\\
r_{10}&\vec{0}&0&0&\vec{0}&2&0&\vec{0}&0&\vec{0}\\
r_{11}&\vec{0}&0&0&\vec{c}& 0&0&\vec{0}&0&\vec{0} \\
\hline
a_3 &\vec{0}& 0 & 0 & \vec{0} & 0&0&\vec{a}&0&\vec{0} \\
b_3 &\vec{0}& 0 & 0 & \vec{0} & 0&0&\vec{b}&0&\vec{0} \\
r_{14}&\vec{0}&0&0&\vec{0}&0&2&\vec{e}&2&\vec{0}\\
r_{15}&\vec{0}&0&0&\vec{0}&0&0&\vec{2}&2&\vec{0}\\
r_{16}&\vec{0}&0&0&\vec{0}&0&0&\vec{0}&2&\vec{0}\\
r_{17}&\vec{0}&0&0&\vec{0}& 0&0&\vec{c}&0&\vec{0} \\
\hline
a_4 &\vec{0}& 0 & 0 & \vec{0} & 0&0&\vec{0}&0&\vec{a} \\
b_4&\vec{0}& 0 & 0 & \vec{0} & 0&0&\vec{0}&0&\vec{b} \\
r_{20}&\vec{0}&0&0&\vec{0}& 0&0&\vec{0}&0&\vec{c} \\
r_{21}&\vec{0}&0&0&\vec{0}& 0&0&\vec{0}&0&\vec{2} \\
\end{array}\end{equation}

Thus for a surface of  genus $4$  the determinant of (\ref{matr}) equals $-2^{13}$.
The procedure can be repeated as many times as needed in the general case, when the matrix corresponds to a surface of larger genus $g$, on the rest of the blocks corresponding to middle handles.   In general for a surface of genus $g$ with one puncture the determinant is equal to $-2^{4g-3}$.   

Since $2$ is relatively prime to all odd numbers, the matrix of resides mod $N$ is invertible.  That implies that the residues of the curves  $P\cup Q\cup\{\partial_1\}$ are a basis for $\mathbb{Z}_N^E$, and we can apply Theorem \ref{basecase}.
 This concludes the proof for surfaces of genus $g>1$ with one puncture.

If the surface has more than one puncture we put all the additional punctures on the bottom handle. We
slice the bottom handle horizontally into pieces.  We continue to build modularly, so if the surface has genus $g$ we have one top handle, $g-2$ middle handles, one bottom handle, and then a planar surface. The planar surface is decomposed into a family of annuli with a point  (a puncture) removed from each boundary component, and a disc with one point removed from its boundary and one from its interior. The  bottom handle is triangulated in the same way as the middle handles. The planar surface is triangulated as a union of triangulations of the annuli and the punctured disk. The  annuli are triangulated with two interior edges, which we denote by $v_k$ and $d_k$ along with the edges in their boundary curves, denoted $b_{k-1},$ and $b_k$. The punctured disk  is triangulated by one folded triangle with a folded edge denoted $f_o$ and the boundary edge $b_j$. The triangulation of an annular piece is shown in Figure \ref{annulus}, and the triangulated punctured disk is pictured in Figure  \ref{disk}.

\begin{figure}[H]\begin{center}\begin{picture}(82,84)\scalebox{1}{\includegraphics{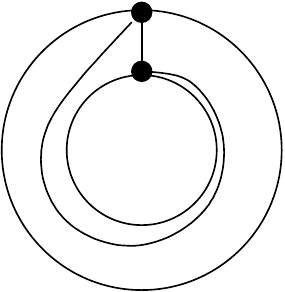}}\put(-40,70){$v_k$}\put(-80,48){$d_k$}\put(-44,25){$b_k$}\end{picture}\end{center}
\caption{Triangulation of an annular piece}\label{annulus}\end{figure}

\begin{figure}[H]\begin{center}\begin{picture}(80,83)\scalebox{1}{\includegraphics{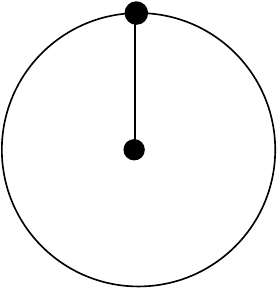}}\put(-38,55){$fo$}\end{picture}\end{center}
\caption{Triangulation of a punctured disk}\label{disk}\end{figure}

We need to add more curves to the families $P$ and $Q$ in order to have  pants decompositions of a multiply punctured surface. We add one curve to each family for each additional puncture. 
For a surface $F$ of genus $g\geq 2$ with one puncture we have 
\begin{equation}P=\{ a_1,\dots a_g,c_1,\dots c_{g-1}, h_1,\dots h_{g-1} \},\end{equation}
and 
\begin{equation}
Q= \{ b_1,\dots b_g, e_1,\dots e_{g-1},s_1,\dots, s_{g-1} \}
\end{equation}
as described above and shown in Figures \ref{P3} and \ref{Q3} for a surface of genus $3$.
For $p$ punctures we add $p-1$ additional curves to each family. The curves $\{c_g, x_1,\dots, x_{p-2}\}$ added to the family $P$ can be described as follows. 
Curve $c_g$ is a parallel copy of the curve $c_{g-1}$ slid across the last puncture so that it does not intersect any other curves in family $P$. Curves 
$\{x_1,\dots,x_{p-2}\}$ are a nested family  that surround an increasing number of punctures, with $x_1$ containing $2$ punctures, and $x_{p-2}$ containing all but the last puncture. One can think of curves $x_i$ as parallel copies of the curve $\partial_1$ around the original puncture, slid across the subsequent punctures. Figure \ref {Ppunct} shows the parts of those curves lying in the bottom handle and in the planar surface, for a surface with $4$ punctures.

\begin{figure}[H]\begin{center}\begin{picture}(144,132)\includegraphics{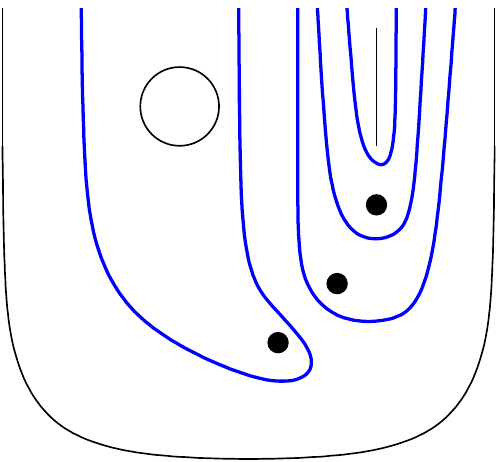}\put(-90,37){$c_g$}\put(-25,50){$x_2$}\put(-40,55){$x_1$}\put(-43,130){$\partial_1$}\end{picture}\end{center}
\caption{Additional curves in the $P$ family for $4$ punctures}\label{Ppunct}
\end{figure}
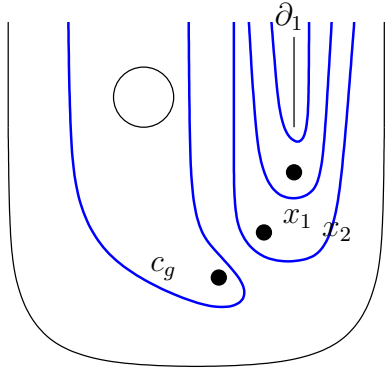

We also add curves $\{s_g,u_1,\dots,u_{p-2}\}$ to the family $Q$. The curve $s_g$ is a parallel copy of the curve $s_{g-1}$ slid across the first puncture. Curves $u_1,\dots,u_{p-2}$ are a nested family  surrounding a decreasing number of punctures, starting with all the additional $p-1$ punctures, all the way down to only containing the last two punctures. All the $u_i$ curves lie in the planar surface, and are shown in Figure \ref{Qpunct} for a surface with $4$ punctures.

\begin{figure}[H]\begin{center}\begin{picture}(144,132)\put(70,102){$s_g$}\put(70,75){$u_1$}\put(70,55){$u_2$}\includegraphics{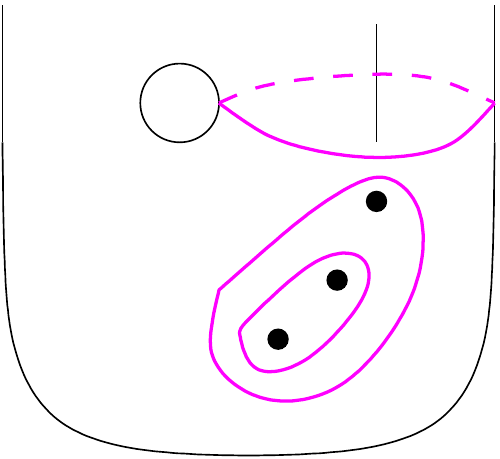}\end{picture}\end{center}
\caption{Additional curves in the $Q$ family for $4$ punctures}\label{Qpunct}
\end{figure}

 If there are just two punctures the argument is a simple extension of what has already been done, so we focus on the case where there are at least three punctures. The proof is similar to the one for a surface with one puncture.
 We calculate the determinant of a matrix whose entries are the geometric intersection numbers of the curves $P\cup Q\cup\{\partial_1,\dots,\partial_p\}$ with the edges of the triangulation. It turns out that this is again a power of $2$. This implies that the change of basis matrix from the standard basis for ${\mathbb{Z}}_N^E$ to the basis consisting of the residues of the admissible colorings corresponding to curves $P\cup Q\cup\{\partial_1,\dots,\partial_p\}$ is invertible.

  The size of the matrix  for a surface of genus $g$ with $p$ punctures is $(6g+3p-6)\times (6g+3p-6)$. We divide the matrix into blocks corresponding to the handles from top down as we did for a surface with one puncture. We have  $6$ columns for the bottom handle,  and  $3p-5$ columns for the planar surface. Since the bottom handle looks now like a middle handle, its block of columns consists of edges $f_g, \vec{h}_g,b_g$, rather than just $\vec{h}_g$. The subsequent columns correspond to the edges triangulating the annular regions from top down: $d_1,v_1,b_{g+1}$,...,$d_{p-2},v_{p-2},b_{g+p-2}$ and finally the folded edge $fo$. 
 The rows corresponding to the bottom handle are ordered as follows:
 $a_g$, $b_g$, $c_{g-1}$, followed by the rows corresponding to the new vectors $s_g$, $c_g$ and the curve surrounding the top puncture $\partial_1$. 
 We group the two families of additional nested curves surrounding the increasing numbers of punctures with the single curves surrounding the additional punctures in one block: $u_1,\dots,u_{p-2}$, $x_1,\dots,x_{p-2}$, and $\partial_2,\dots,\partial_{p}$.
 
 Notice that the curves in the top $g-2$ blocks corresponding to the top $g-2$ handles do not intersect any of the edges starting with $f_g$ thus their entries in the last $3p+1$ columns are all $0$. Recall also that the elementary row operations that reduced that part of the matrix to the triangular block form did not involve any rows corresponding to the bottom handle.
 Thus we can concentrate only on the blocks of rows corresponding to the last two handles and the planar surface: the block of $6$ rows for the $g-1$-st handle, the block of $6$ rows for the bottom handle and the block of rows corresponding to the $u_i,x_i$ and $\partial_i$ curves. 
 
The block of rows for the $g-1$-st handle is the same as for a surface with $1$ puncture, with zeroes filling up the additional columns. The rows for the  bottom handle are as follows: the rows for the $a_g$ and the $b_g$ curves have additional zeroes in the new columns; the row for the $c_{g-1}$ curve has the same pattern as the rows  $c_i$ with $i<g-1$; the row for the curve $s_g$ has zeroes in the columns $\vec{h}_1$ through $b_{g-1}$, 
then $(2,\vec{e}, 2,1,1)$ in the columns $ (f_g,\vec{h}_g, b_g,d_1,v_1)$ and zeroes in the remaining columns; the row for the curve $c_g$ is the same as for the curve $c_{g-1}$ up to the column $b_{g-1}$, then has $(4,\vec{f},2)$ in the columns $(f_g,\vec{h}_g,b_g)$, and the vector $(2,0,2)$ in each sequence of columns $(d_i,v_i,b_{g+i})$, ending with $1$ in the column $f_o$; the row for the curve $\partial_1$ surrounding the top puncture has twos in the columns $\vec{h}_1$ through $b_g$, $(1,1)$ in columns $(d_1,v_1)$ and zeroes in the remaining columns.

The rows in the block corresponding to the planar surface are very simple. The row  for the curve $u_i$ has $1$ in the columns $d_i$ and $v_i$, and zeroes in the remaining columns.
The row for the curve $x_i$ has $2$ in columns $\vec{h}_1$ up to $b_g$, followed by $(2,0,2)$  in columns $(d_j,v_j,b_{g+j})$ for $j=1,\dots,i$, ending with $(1,1)$ in columns $(d_{i+1}, v_{i+1})$ and $0$ in the subsequent columns, unless $i=p-2$, in which case the last entry is in the column $f_o$ and it is equal to $1$.
The row for the curve $\partial_i$ surrounding the $i$-th puncture for $i=2,\dots,p-2$ has $(1,1,2,1,1,0)$ in columns $(d_{i-1},v_{i-1},b_{g+1},d_i,v_i,b_{g+2})$ and zeroes in all other columns, row vector for the curve $\partial_{p-1}$ has $(1,1,2,1)$ in the last four columns $(d_{p-2},v_{p-2}, b_{g+p-2},f_o)$ and zeroes in all preceding columns, finally the row vector $\partial_p$  for the curve surrounding the bottom puncture has $1$ in the column $f_o$ and zeroes in all other columns.

Observe that we can perform the same iterative procedures as for the matrix for a surface with $1$ puncture to clear all entries up to column $b_{g-2}$ for all rows in the  blocks corresponding to the last middle handle and to the bottom handle. We can also use elementary row operations to clear out the twos appearing in columns $\vec{h}_1,\dots, b_{g-2}$ in row vectors for all the $x_i$ curves, since row vectors for the top $g-2$ handles are linearly independent and have zeroes appearing in columns $f_{g-1}$ and higher. Furthermore, we can assume that we performed the necessary elementary row operations so  that the rows of the block corresponding to the $g-1$ handle look like the analogous rows of the matrix (\ref{blockmatrix}).

We write out the relevant blocks of the matrix for a surface with $4$ punctures and use this example to illustrate the general procedure to reduce the matrix to block form. Since we are assuming that the row vectors for curves $s_{g-1}$, $h_{g-1}$ and $c_{g-2}$ have been modified already, we label those $r_4$, $r_5$ and $r_6$, counting just the rows appearing on the page.

\begin{equation}\begin{array}{l|ccc|ccc|ccccccc}\label{pmatr}
\ & f_{g-1} & \vec{h}_{g-1} & b_{g-1} & f_g & \vec{h}_g & b_g & d_1 & v_1 & b_{g+1} & d_2 & v_2 & b_{g+2} & fo \\
a_{g-1} & 0 & \vec{a} & 0 & 0 & \vec{0} & 0 & 0 & 0 & 0 & 0 & 0 & 0 & 0 \\
b_{g-1} & 0 & \vec{b} & 0 & 0 & \vec{0} & 0 & 0 & 0 & 0 & 0 & 0 & 0 & 0 \\
e_{g-1} & 2 & \vec{e} & 2 & 2 &   \vec{0} & 0 & 0 & 0 & 0 & 0 & 0 & 0 & 0 \\
r_4 & 0 & \vec{2} & 2 & 2&  \vec{0} & 0 & 0 & 0 & 0 & 0 & 0 & 0 & 0 \\
r_5 & 0 & \vec{0} & 2 & 2&  \vec{0} & 0 & 0 & 0 & 0 & 0 & 0 & 0 & 0 \\
r_6 & 0 & \vec{c}& 0 & 0 &  \vec{0} & 0 & 0 & 0 & 0 & 0 & 0 & 0 & 0 \\
\hline
a_g& 0 &\vec{0} & 0 & 0 &\vec{a}& 0 & 0 & 0 & 0 & 0 & 0 & 0 & 0 \\
b_g& 0 & \vec{0} & 0 & 0 &\vec{b}& 0 & 0 & 0 & 0 & 0 & 0 & 0 & 0 \\
s_g& 0 & \vec{0} & 0 & 2 &\vec{e}& 2& 1&1 & 0 & 0 & 0 & 0 & 0 \\
c_{g-1}& 4 & \vec{f} & 2 & 2 &\vec{c}& 0& 0&0 & 0 & 0 & 0 & 0 & 0 \\
c_{g}& 4 & \vec{f} & 2 & 4 &\vec{f}& 2& 2&0 & 2 & 2 & 0 & 2 & 1 \\
\partial_1& 2 & \vec{2} & 2 & 2 &\vec{2}& 2 & 1& 1 & 0 & 0 & 0 & 0 & 0 \\
\hline
u_1&  0 & \vec{0} & 0 & 0 & \vec{0} & 0& 1 & 1 & 0 & 0 & 0 & 0 & 0 \\
u_2&  0 &\vec{0} & 0&  0 &  \vec{0} & 0 & 0 & 0 & 0& 1 & 1 & 0 & 0  \\
x_1& 2 & \vec{2} & 2 & 2 &\vec{2}& 2 & 2& 0 & 2 & 1& 1 & 0 & 0 \\
x_2& 2 & \vec{2} & 2 & 2 &\vec{2}& 2 & 2& 0 & 2 & 2& 0 & 2 & 1 \\
\partial_2& 0 & \vec{0} & 0 & 0 & \vec{0} & 0& 1 & 1 & 2 & 1 & 1 & 0 & 0 \\
\partial_3& 0 & \vec{0} & 0 & 0 & \vec{0} & 0& 0 & 0 & 0 & 1 & 1 & 2 & 1 \\
\partial_4& 0 & \vec{0} & 0 & 0 & \vec{0} & 0& 0 & 0 & 0 & 0 & 0 & 0 & 1 \\
\end{array}\end{equation}

We are only interested in proving that the determinant of this matrix is $\pm 2^k$ for some $k$.  It is clear that if we delete the last row and the last column we will at worst change the sign of the determinant, so we do that to simplify the picture. We also  replace row $c_g$ by $c_g-x_2-a_{g-1}-a_g$ and $c_{g-1}$ by
$c_{g-1}-e_{g-1}+b_{g-1}-r_6$. Using the fact that $\vec{a}+\vec{2}=\vec{f}$, and
$\vec{f}=\vec{e}-\vec{b}+\vec{c}$ we get the following matrix, where we use the previous convention of renaming the row with its number on the page once it has been changed.

\begin{equation}\begin{array}{l|ccc|ccc|cccccc}\label{pmatr2}
\ & f_{g-1} & \vec{h}_{g-1} & b_{g-1} & f_g & \vec{h}_g & b_g & d_1 & v_1 & b_{g+1} & d_2 & v_2 & b_{g+2} \\
a_{g-1} & 0 & \vec{a} & 0 & 0 & \vec{0} & 0 & 0 & 0 & 0 & 0 & 0 & 0 \\
b_{g-1} & 0 & \vec{b} & 0 & 0 & \vec{0} & 0 & 0 & 0 & 0 & 0 & 0 & 0  \\
e_{g-1} & 2 & \vec{e} & 2 & 2 &   \vec{0} & 0 & 0 & 0 & 0 & 0 & 0 & 0 \\
r_4 & 0 & \vec{2} & 2 & 2&  \vec{0} & 0 & 0 & 0 & 0 & 0 & 0 & 0  \\
r_5 & 0 & \vec{0} & 2 & 2&  \vec{0} & 0 & 0 & 0 & 0 & 0 & 0 & 0  \\
r_6 & 0 & \vec{c}& 0 & 0 &  \vec{0} & 0 & 0 & 0 & 0 & 0 & 0 & 0  \\
\hline
a_g& 0 &\vec{0} & 0 & 0 &\vec{a}& 0 & 0 & 0 & 0 & 0 & 0 & 0  \\
b_g& 0 & \vec{0} & 0 & 0 &\vec{b}& 0 & 0 & 0 & 0 & 0 & 0 & 0  \\
s_g& 0 & \vec{0} & 0 & 2 &\vec{e}& 2& 1&1 & 0 & 0 & 0 & 0  \\
r_{10}& 2 & \vec{0} & 0 & 2 &\vec{c}& 0& 0&0 & 0 & 0 & 0 & 0  \\
r_{11}& 2 & \vec{0} & 0 & 2 &\vec{0}& 0& 0&0 & 0 & 0 & 0 & 0  \\
\partial_1& 2 & \vec{2} & 2 & 2 &\vec{2}& 2 & 1& 1 & 0 & 0 & 0 & 0  \\
\hline
u_1&  0 & \vec{0} & 0 & 0 & \vec{0} & 0& 1 & 1 & 0 & 0 & 0 & 0 \\
u_2&  0 &\vec{0} & 0&  0 &  \vec{0} & 0 & 0 & 0 & 0& 1 & 1 & 0   \\
x_1& 2 & \vec{2} & 2 & 2 &\vec{2}& 2 & 2& 0 & 2 & 1& 1 & 0 \\
x_2& 2 & \vec{2} & 2 & 2 &\vec{2}& 2 & 2& 0 & 2 & 2& 0 & 2  \\
\partial_2& 0 & \vec{0} & 0 & 0 & \vec{0} & 0& 1 & 1 & 2 & 1 & 1 & 0  \\
\partial_3& 0 & \vec{0} & 0 & 0 & \vec{0} & 0& 0 & 0 & 0 & 1 & 1 & 2  \\
 
\end{array}\end{equation}

Notice that $\vec{e}=\vec{2}+\vec{a}+\vec{b}-\vec{c}$  implies that $e_{g-1}-r_4-a_{g-1}-b_{g-1}$ is the vector that has a $2$ in the $f_{g-1}$ column and zeroes everywhere else.  We subtract this combination from rows $r_{10}$ and $r_{11}$, and also subtract this combination and row $r_4$ from row $\partial_1$ to get the following matrix.

\begin{equation}\begin{array}{l|ccc|ccc|cccccc}\label{pmatrix3}
\ & f_{g-1} & \vec{h}_{g-1} & b_{g-1} & f_g & \vec{h}_g & b_g & d_1 & v_1 & b_{g+1} & d_2 & v_2 & b_{g+2} \\
a_{g-1} & 0 & \vec{a} & 0 & 0 & \vec{0} & 0 & 0 & 0 & 0 & 0 & 0 & 0 \\
b_{g-1} & 0 & \vec{b} & 0 & 0 & \vec{0} & 0 & 0 & 0 & 0 & 0 & 0 & 0  \\
e_{g-1} & 2 & \vec{e} & 2 & 2 &   \vec{0} & 0 & 0 & 0 & 0 & 0 & 0 & 0 \\
r_4 & 0 & \vec{2} & 2 & 2&  \vec{0} & 0 & 0 & 0 & 0 & 0 & 0 & 0  \\
r_5 & 0 & \vec{0} & 2 & 2&  \vec{0} & 0 & 0 & 0 & 0 & 0 & 0 & 0  \\
r_6 & 0 & \vec{c}& 0 & 0 &  \vec{0} & 0 & 0 & 0 & 0 & 0 & 0 & 0  \\
\hline
a_g& 0 &\vec{0} & 0 & 0 &\vec{a}& 0 & 0 & 0 & 0 & 0 & 0 & 0  \\
b_g& 0 & \vec{0} & 0 & 0 &\vec{b}& 0 & 0 & 0 & 0 & 0 & 0 & 0  \\
s_g& 0 & \vec{0} & 0 & 2 &\vec{e}& 2& 1&1 & 0 & 0 & 0 & 0  \\
r_{10}& 0 & \vec{0} & 0 & 2 &\vec{c}& 0& 0&0 & 0 & 0 & 0 & 0  \\
r_{11}& 0& \vec{0} & 0 & 2 &\vec{0}& 0& 0&0 & 0 & 0 & 0 & 0  \\
r_{12}& 0 & \vec{0} & 0 & 0 &\vec{2}& 2 & 1& 1 & 0 & 0 & 0 & 0  \\
\hline
u_1&  0 & \vec{0} & 0 & 0 & \vec{0} & 0& 1 & 1 & 0 & 0 & 0 & 0 \\
u_2&  0 &\vec{0} & 0&  0 &  \vec{0} & 0 & 0 & 0 & 0& 1 & 1 & 0   \\
x_1& 2 & \vec{2} & 2 & 2 &\vec{2}& 2 & 2& 0 & 2 & 1& 1 & 0 \\
x_2& 2 & \vec{2} & 2 & 2 &\vec{2}& 2 & 2& 0 & 2 & 2& 0 & 2  \\
\partial_2& 0 & \vec{0} & 0 & 0 & \vec{0} & 0& 1 & 1 & 2 & 1 & 1 & 0  \\
\partial_3& 0 & \vec{0} & 0 & 0 & \vec{0} & 0& 0 & 0 & 0 & 1 & 1 & 2  \\

\end{array}\end{equation}

Notice that $r_{11}$ has a single nonzero entry in column $f_g$ so we can use it to clear the twos in this column from rows $e_{g-1}$, $r_4$ and $r_5$. Also we can use row $u_1$ to clear out the nonzero entries in rows $s_g$ and $r_{12}$  in the columns $d_1$ and $v_1$, so that the resulting matrix (\ref{pmatr4}) below is in block lower triangular form.

\begin{equation}\begin{array}{l|ccc|ccc|cccccc}\label{pmatr4}
\ & f_{g-1} & \vec{h}_{g-1} & b_{g-1} & f_g & \vec{h}_g & b_g & d_1 & v_1 & b_{g+1} & d_2 & v_2 & b_{g+2} \\
a_{g-1} & 0 & \vec{a} & 0 & 0 & \vec{0} & 0 & 0 & 0 & 0 & 0 & 0 & 0 \\
b_{g-1} & 0 & \vec{b} & 0 & 0 & \vec{0} & 0 & 0 & 0 & 0 & 0 & 0 & 0  \\
e_{g-1} & 2 & \vec{e} & 2 & 0&   \vec{0} & 0 & 0 & 0 & 0 & 0 & 0 & 0 \\
r_4 & 0 & \vec{2} & 2 & 0&  \vec{0} & 0 & 0 & 0 & 0 & 0 & 0 & 0  \\
r_5 & 0 & \vec{0} & 2 & 0&  \vec{0} & 0 & 0 & 0 & 0 & 0 & 0 & 0  \\
r_6 & 0 & \vec{c}& 0 & 0 &  \vec{0} & 0 & 0 & 0 & 0 & 0 & 0 & 0  \\
\hline
a_g& 0 &\vec{0} & 0 & 0 &\vec{a}& 0 & 0 & 0 & 0 & 0 & 0 & 0  \\
b_g& 0 & \vec{0} & 0 & 0 &\vec{b}& 0 & 0 & 0 & 0 & 0 & 0 & 0  \\
s_g& 0 & \vec{0} & 0 & 2 &\vec{e}& 2& 0&0 & 0 & 0 & 0 & 0  \\
r_{10}& 0 & \vec{0} & 0 & 2 &\vec{c}& 0& 0&0 & 0 & 0 & 0 & 0  \\
r_{11}& 0& \vec{0} & 0 & 2 &\vec{0}& 0& 0&0 & 0 & 0 & 0 & 0  \\
\partial_1& 0 & \vec{0} & 0 & 0 &\vec{2}& 2 & 0& 0 & 0 & 0 & 0 & 0  \\
\hline
u_1&  0 & \vec{0} & 0 & 0 & \vec{0} & 0& 1 & 1 & 0 & 0 & 0 & 0 \\
u_2&  0 &\vec{0} & 0&  0 &  \vec{0} & 0 & 0 & 0 & 0& 1 & 1 & 0   \\
x_1& 2 & \vec{2} & 2 & 2 &\vec{2}& 2 & 2& 0 & 2 & 1& 1 & 0 \\
x_2& 2 & \vec{2} & 2 & 2 &\vec{2}& 2 & 2& 0 & 2 & 2& 0 & 2  \\
\partial_2& 0 & \vec{0} & 0 & 0 & \vec{0} & 0& 1 & 1 & 2 & 1 & 1 & 0  \\
\partial_3& 0 & \vec{0} & 0 & 0 & \vec{0} & 0& 0 & 0 & 0 & 1 & 1 & 2  \\

\end{array}\end{equation}

This means that the determinant of the matrix is the product of the determinants of the blocks on the diagonal.  The blocks corresponding to the top $g-1$ handles have determinant $\pm2^k$ for some $k$. The determinant of the $g$-th block is $-2^4$. 

If the surface has $p$ punctures then the bottom diagonal block is a $3(p-2)\times 3(p-2)$ matrix consisting of vectors $u_i$, $x_i$ and $\partial_i$, which were described above, starting with their entries in column $d_1$.  We show this block for $p=5$, and it is simple to see the general pattern. 
\begin{equation}\begin{pmatrix} 1 & 1& 0 & 0 & 0 & 0 & 0 & 0 & 0 \\
                                0 & 0 & 0& 1 & 1& 0 & 0 & 0 & 0 \\
                                0 & 0 & 0 & 0 & 0 & 0& 1 & 1 & 0 \\
                                2 & 0 & 2 & 1 & 1 & 0 & 0 & 0 & 0\\
                                2 & 0 & 2 & 2 & 0 & 2 &1 & 1 & 0 \\
                                2 & 0 & 2 & 2 & 0 & 2 & 2 & 0 & 2  \\
                                1 & 1 & 2 & 1 & 1 & 0 & 0 & 0 & 0 \\
                                0 & 0 & 0 & 1 & 1 & 2 & 1 & 1 & 0 \\
                                0 & 0 & 0 &  0 & 0 & 0 &1& 1 & 2 \end{pmatrix} \end{equation}
It is an easy exercise to see that the determinant of this matrix is $\pm2^{2p-4}$. 

Hence the determinant of the whole matrix is equal $\pm2^K$ for some $K$, and is a unit $\mod{N}$ as desired.
We apply Theorem \ref{basecase} to finish the proof.
 \qed

\end{document}
